\documentclass[11pt]{article}
\pdfoutput=1

\usepackage[dvipsnames]{xcolor}

\usepackage{caption}

\usepackage{indentfirst}
\usepackage{bm, mathrsfs, graphics,float,amssymb,amsmath,subeqnarray,setspace,graphicx,amsthm,epstopdf,subfigure, enumerate, color}
\usepackage[utf8]{inputenc}
\usepackage[colorlinks,
            linkcolor=red,
            anchorcolor=blue,
            citecolor=blue
            ]{hyperref}
\usepackage{natbib}

\newcommand{\papertitle}{High-Probability Last-Iterate Guarantees for Two-Point Gaussian Zeroth-Order Stochastic Gradient Descent}
\hypersetup{
	pdftitle={High-Probability Last-Iterate Guarantees for Two-Point Gaussian Zeroth-Order Stochastic Gradient Descent},
	pdfauthor={Haishan Ye},
	pdfkeywords={zeroth-order stochastic optimization, derivative-free optimization, stochastic approximation, high-probability analysis, sub-Gaussian noise, strong convexity}
}

\usepackage{fullpage}
\numberwithin{equation}{section}

\newtheorem{lemma}{Lemma}
\newtheorem{theorem}{Theorem}
\newtheorem{proposition}{Proposition}
\newtheorem{corollary}[theorem]{Corollary}
\newtheorem{remark}{Remark}
\ifx\assumption\undefined
\newtheorem{assumption}{Assumption}
\fi

\newcommand{\cE}{\mathcal{E}}
\newcommand{\EE}{\mathbb{E}}
\newcommand{\RR}{\mathbb{R}}
\newcommand{\bg}{\bm{g}}
\newcommand{\ba}{\bm{a}}
\newcommand{\be}{\bm{e}}
\newcommand{\bx}{\bm{x}}
\newcommand{\by}{\bm{y}}
\newcommand{\bu}{\bm{u}}
\newcommand{\bv}{\bm{v}}
\newcommand{\bz}{\bm{z}}
\newcommand{\cF}{{{\mathcal{F}}}}
\newcommand{\PP}{\mathbb{P}}
\newcommand{\cC}{\mathcal{C}}
\newcommand{\proofstep}[1]{\par\medskip\noindent\textit{#1}\ }

\usepackage{algorithm}
\usepackage{algorithmic}

\begin{document}
\title{\papertitle}

\author{
Haishan Ye
\thanks{
	Xi'an Jiaotong University;
	email: hsye\_cs@outlook.com
}
}
\date{\today}

\maketitle

\newcommand{\norm}[1]{\left\|#1\right\|}
\newcommand{\dotprod}[1]{\left\langle #1\right\rangle}

\begin{abstract}
We establish a direct high-probability last-iterate guarantee for the standard
same-sample two-point Gaussian zeroth-order SGD method in smooth, strongly
convex stochastic optimization. At each iteration, the method draws a fresh
Gaussian direction, evaluates two symmetric perturbations with the same
stochastic sample, and takes a norm-normalized stochastic approximation step.
Assuming unbiased stochastic gradients and a conditional exponential-moment
bound on the squared norm of the stochastic gradient noise, we prove a
finite-horizon bound, valid for dimension \(d\ge2\), with an explicit
product-weight factor. When the offset in the stepsize schedule is large enough
relative to the logarithmic confidence terms, this factor is bounded and the
result gives
\[
f(\bx_T)-f(\bx^*)
=
\widetilde{\mathcal O}\!\left(\frac{d}{T}\right)
\]
with probability at least \(1-\delta\), up to fixed problem parameters and
logarithmic factors. Thus the confidence dependence is logarithmic rather than
polynomial in \(1/\delta\), and the proof neither invokes Markov's inequality
nor truncates the noise. To the best of our knowledge, this is the first direct
high-probability last-iterate result at this zeroth-order scale for the
same-sample Gaussian recursion under conditional sub-Gaussian
stochastic-gradient noise. The proof combines uniform weighted lower and upper
scans for Gaussian angles, a product-martingale boundary for the signed
suffix-product term, and terminal nonnegative concentration estimates. We also
formulate the resulting general pathwise framework for stochastic recursions
with random contraction and signed perturbations, identifying the scan,
filtration, variance, and terminal-control conditions under which the same
argument applies.
\end{abstract}


\section{Introduction}

We consider stochastic optimization problems in which gradients are unavailable,
unreliable, or prohibitively expensive. This regime arises in simulation
optimization \citep{nemirovski2009robust,ghadimi2013stochastic}, black-box and
derivative-free optimization \citep{conn2009introduction,nesterov2017random},
bandit learning \citep{flaxman2005online,shamir2017optimal}, stochastic
approximation with oracle feedback \citep{robbins1951stochastic}, and
memory-constrained model adaptation \citep{malladi2023fine}. These settings
call for finite-sample guarantees that make both the dimension and the
confidence level explicit. We consider
\begin{equation}\label{eq:f_sto}
	\min_{\bx\in\mathbb{R}^d} f(\bx)
	=
	\EE_\xi[f(\bx;\xi)]
\end{equation}
using only stochastic function values. A query at \(\bx\) returns
\(f(\bx;\xi)\) for a random sample \(\xi\). Within each two-point finite
difference, both evaluations use the same sample. The resulting randomized
finite-difference estimator can be inserted into a gradient-descent-type
recursion \citep{conn2009introduction,nesterov2017random,duchi2015optimal}. In
many simulation and black-box applications, however, the method is run once and
the terminal decision is deployed. A high-probability guarantee for that last
iterate is therefore more informative than an expected guarantee interpreted
through hypothetical independent repetitions.

Classical analyses of stochastic zeroth-order methods primarily control expected
suboptimality. For smooth, strongly convex objectives, they yield the familiar
\(\mathcal O(d/T)\) rate, up to problem-dependent constants
\citep{ghadimi2013stochastic,nesterov2017random,malladi2023fine,wang2023fine}.
Expectation alone does not certify the performance of a single run. Indeed, if
\(\EE[f(\bx_T)-f(\bx^*)]\le r(T)\), then Markov's inequality gives
\(\PP\{f(\bx_T)-f(\bx^*)>\varepsilon\}\le\delta\) only when
\(r(T)\le\delta\varepsilon\). Consequently, an \(\mathcal O(d/T)\) expected
rate yields only an \(\mathcal O(d/(\delta T))\) confidence bound
through this generic conversion. A direct concentration analysis should instead
incur only logarithmic dependence on \(1/\delta\), as sought in robust
stochastic approximation and large-deviation theory
\citep{nemirovski2009robust,ghadimi2013stochastic,gorbunov2021high}. In the
zeroth-order setting, a Markov conversion may additionally force the smoothing
radius to depend on the confidence level, obscuring the behavior of the
standard algorithmic parameterization \citep{ye2026smoothhpzo}.

A direct high-probability analysis faces a geometric obstacle absent from the
usual expectation argument. Even in the noiseless case, the descent term
contains the random alignment \((\bu_t^\top\nabla f(\bx_t))^2\). Its
conditional mean is proportional to \(\|\nabla f(\bx_t)\|^2\), which is
sufficient after taking expectations. Pathwise, however, the normalized
alignment has a one-dimensional \(\chi^2\)-type lower tail and cannot be
bounded away from zero at every iteration with high probability. Thus no
uniform deterministic contraction is available. The analysis must instead
accumulate contraction across many directions and control the accumulated
alignment simultaneously over the relevant suffixes of the trajectory.

The stochastic oracle creates a second obstacle. The finite-difference
estimator contains both smoothing bias and stochastic error. After the descent
recursion is unrolled, the linear error generated at time \(k\) is multiplied
by future random contractions determined by \(\{\bu_t\}_{t>k}\), producing a
signed adaptive suffix-product martingale. Replacing this product by a
deterministic envelope, or taking absolute values before applying
concentration, loses the cancellation needed for a sharp rate. At the same
time, the associated variance proxy couples weighted Gaussian projections with
quadratic-noise and smoothing terms. These components must be controlled on a
common event, uniformly over the terminal time.

Although our main theorem is stated for a zeroth-order Gaussian method with
same-sample stochastic-gradient noise, the proof exposes a more general
pathwise structure. Once an algorithmic error process can be written as a
random-contraction recursion, the roles of the contraction law, the signed
perturbation term, and the nonnegative terminal terms can be separated. Thus
the same method applies beyond the particular zeroth-order recursion whenever
the corresponding scan, filtration, variance, and terminal-control conditions
can be verified. Other sketch distributions and additive function-value noise
are concrete examples of this modular replacement.

\paragraph{Contributions.}
Our contributions are fivefold.
\begin{enumerate}[(i)]
	\item We derive a direct high-probability bound for the last iterate of the
	standard same-sample two-point Gaussian zeroth-order SGD recursion with the
	norm-normalized stepsize
	\(\eta_t=4d/[\mu(t+T_0)\|\bu_t\|^2]\). The argument neither truncates
	nor clips the noise, and it does not rely on online-to-batch or
	Markov-type conversions. To the best of our knowledge, no earlier direct
	result attains the \(\widetilde{\mathcal O}(d/T)\) scale for this
	recursion under conditional sub-Gaussian stochastic-gradient noise.
	\item The guarantee is dimension-explicit. Under
	the clean-rate condition \(T_0\gtrsim\Gamma_T(\delta)\),
	Corollary~\ref{cor:main_1} yields
	\[
		f(\bx_T)-f(\bx^*)
		=
		\widetilde{\mathcal O}\!\left(\frac{d}{T}\right)
	\]
	with probability at least \(1-\delta\), after fixed problem
	parameters and logarithmic factors. Without this simplification, the
	finite-horizon theorem keeps the explicit product-weight factor
	\(c_\rho=\exp\{O(\Gamma_T(\delta)/T_0)\}\).
	\item We give a proof architecture for the unrolled recursion that is uniform
	in the terminal time. Weighted lower and upper scans control accumulated
	Gaussian alignment over all relevant suffixes and prefixes, replacing an
	unavailable per-step contraction by aggregate contraction and providing the
	energy bounds needed by variance envelopes. A stitched product-martingale
	boundary then controls the signed linear term while retaining the realized
	product weights until cancellation has been used. Terminal concentration
	estimates control the quadratic-noise and smoothing-bias terms.
	\item We isolate the pathwise framework behind the proof. The framework is
	organized around a generic random-contraction stochastic recursion, strict
	product stability, predictable weighted scans, a martingale transformation
	for the signed perturbation, and terminal nonnegative controls. This
	formulation clarifies which parts of the proof are structural and which
	conditions must be checked in any new model before the method can be reused.
	\item We illustrate this framework within zeroth-order settings by changing
	the direction law and the oracle perturbation module. Rademacher and random
	coordinate directions require new weighted scan and projection estimates,
	while additive function-value noise replaces the signed and nonnegative
	modules and yields new \(Q_t\), \(B_t\), and intrinsic-variance terms.
\end{enumerate}

The remainder of the paper is organized as follows. We first review related
work, then introduce the algorithm, assumptions, and main results.
Section~\ref{sec:zsgd} develops the descent recursion, the time-uniform
concentration events, and the induction proving the main theorem.
Section~\ref{sec:pathwise_template} presents the general pathwise framework
behind the proof and then illustrates its use for other zeroth-order direction
laws and oracle perturbation modules. The appendix collects the auxiliary
concentration inequalities and detailed proofs.

\subsection{Related Work}

\paragraph{Derivative-free and bandit convex optimization.}
Zeroth-order and derivative-free methods are classical tools in continuous
optimization \citep{conn2009introduction}. Gaussian- and sphere-smoothed
randomized methods were analyzed by \citet{nesterov2017random}, and
\citet{duchi2015optimal} established minimax rates for zeroth-order convex
optimization. In online convex optimization, \citet{flaxman2005online}
introduced one-point bandit estimators, with subsequent work obtaining sharper
rates from two-point feedback \citep{agarwal2010optimal,shamir2017optimal}.
Recently, \citet{yu2026improved} showed that online convex optimization with
two-point feedback can achieve the optimal expected regret for strongly convex
losses, with optimal dependence on both the horizon and the dimension.
High-probability two-point bandit results, including \citet{ye2026optimal},
address related concentration questions but operate under different oracle
models, output criteria, and boundedness assumptions. In particular, an online
regret bound does not directly imply a last-iterate confidence guarantee for the
stochastic-approximation recursion studied here.

\paragraph{High-probability stochastic approximation.}
Stochastic approximation originates with \citet{robbins1951stochastic}, and
confidence guarantees beyond expectation have been studied through robust and
large-deviation analyses
\citep{nemirovski2009robust,ghadimi2013stochastic,gorbunov2021high}. For
first-order SGD, high-probability bounds are well understood under bounded
stochastic gradients or bounded noise \citep{rakhlin2012making}.
\citet{liu2023revisiting} removed an almost-sure bounded-gradient condition and
proved high-probability last-iterate guarantees under sub-Gaussian noise,
obtaining an \(\mathcal O(\log T/T)\) rate without prior knowledge of \(T\). Their
analysis is a natural first-order benchmark. A zeroth-order proof must
additionally handle random search directions, smoothing bias, and stochastic
finite-difference errors.

\paragraph{Stochastic zeroth-order methods.}
Most stochastic zeroth-order analyses establish guarantees in expectation
\citep{ghadimi2013stochastic,duchi2015optimal,shamir2017optimal}. For smooth,
strongly convex objectives, randomized methods attain
\(\mathcal O(d/T)\) expected rates under bounded-variance-type conditions
\citep{wang2023fine,malladi2023fine}. High-probability guarantees are also
available for robustified algorithms. In a nonsmooth heavy-tailed setting,
\citet{kornilov2023acceleratedzo} derive iteration and oracle-complexity bounds
for accelerated clipped schemes. Those results modify the estimator to obtain
robustness; by contrast, we analyze the unmodified same-sample two-point
Gaussian recursion and target its last iterate under conditional sub-Gaussian
stochastic-gradient noise.

Table~\ref{tab:rates_sto} gives a schematic rate comparison. Because the oracle
models and output criteria differ across rows, the table should be read only at
the level of rate order. For orientation, first-order results are displayed on
the customary zeroth-order scale by including a dimension factor; our bound is
proved directly for the two-point Gaussian method.

\begin{table}[tb]
	\centering
	\small
	\resizebox{\textwidth}{!}{%
	\begin{tabular}{lllc}
		\hline
		\textbf{Reference} & \textbf{Guarantee} & \textbf{Noise assumption} & \textbf{Rate after \(T\) iterations} \\ \hline
		\citet{bottou2018optimization} & Expectation & Bounded variance & $\mathcal{O}\!\left(\frac{d}{T}\right)$ \\ \hline
		\citet{rakhlin2012making} & High probability & Bounded gradient & $\widetilde{\mathcal{O}}\!\left(\frac{d}{T}\right)$ \\ \hline
		\citet{liu2023revisiting} & High probability & Sub-Gaussian noise & $\widetilde{\mathcal{O}}\!\left(\frac{d}{T}\right)$ \\ \hline
		\citet{ye2026optimal} & High probability & Bounded gradient & $\mathcal{O}\!\left(\frac{d(\log T+\log(1/\delta))}{T}\right)$ \\ \hline
		\citet{wang2023fine} & Expectation & Bounded variance & $\mathcal{O}\!\left(\frac{d}{T}\right)$ \\ \hline
		\textbf{This paper} & \textbf{High probability} & \textbf{Conditional sub-Gaussian noise} & $\mathcal{O}\!\left(\frac{d\log(T+T_0)\Gamma_T(\delta)}{T+T_0}\right)$ \\ \hline
	\end{tabular}
	}
	\caption{Schematic rate comparison for smooth, strongly convex
	stochastic optimization. The oracle models, assumptions, and output criteria
	differ across rows. First-order results are displayed on a
	zeroth-order scale by including the conventional dimension factor. The entry
	for this paper suppresses fixed problem parameters and the additional factor
	\(1+\Gamma_T(\delta)/T_0\) in Corollary~\ref{cor:main_1}.}
	\label{tab:rates_sto}
\end{table}

Table~\ref{tab:direct_zo_comparison} isolates the closest zeroth-order
comparisons. The salient distinctions concern the oracle model, the output
criterion, whether the guarantee is in expectation or with high probability,
and whether robustness is obtained from boundedness, clipping, or a
sub-Gaussian noise condition.

\begin{table}[tb]
	\centering
	\small
	\resizebox{\textwidth}{!}{%
	\begin{tabular}{lllll}
		\hline
		\textbf{Reference} & \textbf{Oracle model} & \textbf{Output/criterion} & \textbf{Guarantee} & \textbf{Noise or boundedness} \\ \hline
		\citet{ghadimi2013stochastic} & Stochastic zeroth-order & Expected stationarity/suboptimality & Expectation & Bounded variance \\ \hline
		\citet{duchi2015optimal}, \citet{shamir2017optimal} & Two-point zeroth-order/bandit & Minimax or averaged guarantees & Expectation/regret & Model-dependent boundedness \\ \hline
		\citet{wang2023fine} & Gaussian zeroth-order & Strongly convex suboptimality & Expectation & Bounded variance \\ \hline
		\citet{kornilov2023acceleratedzo} & Clipped zeroth-order method & Convex oracle complexity & High probability & Heavy-tailed noise with clipping \\ \hline
		\citet{ye2026optimal} & Two-point bandit feedback & Online regret & High probability & Bounded gradient \\ \hline
		\textbf{This paper} & \textbf{Two-point Gaussian oracle} & \textbf{Last iterate} & \textbf{High probability} & \textbf{Sub-Gaussian noise} \\ \hline
	\end{tabular}
	}
	\caption{Direct comparison with representative zeroth-order results.
	Our theorem concerns the last iterate of the standard same-sample
	two-point Gaussian stochastic-approximation recursion; it should therefore
	not be conflated with expectation bounds, minimax risk guarantees, or online
	regret results.}
	\label{tab:direct_zo_comparison}
\end{table}

\section{Preliminaries and Assumptions}

\subsection{Preliminaries}
We begin with the two-point estimator and the resulting stochastic recursion.
At iteration \(t\), the method draws a fresh Gaussian direction and a fresh
oracle sample; both function evaluations in the finite difference use the same
sample \(\xi_t\).

Given a point \(\bx\), a direction \(\bu\), and a common sample \(\xi\), the
oracle is queried at \(\bx+\alpha\bu\) and \(\bx-\alpha\bu\), and the
estimator is defined by
\begin{equation}\label{eq:sg}
	\bg(\bx;\xi) = \frac{f(\bx+\alpha\bu;\xi) - f(\bx-\alpha\bu;\xi)}{2\alpha}\bu.
\end{equation}

Algorithm~\ref{alg:SA_1} applies this estimator recursively.

\begin{algorithm}[t]
	\caption{Same-Sample Two-Point Gaussian Zeroth-Order SGD}
	\label{alg:SA_1}
	\begin{small}
		\begin{algorithmic}[1]
			\STATE {\bf Input:}
			Deterministic initial point \(\bx_0\) and horizon \(T\).
			\STATE Set \(T_0\) and the smoothing parameter \(\alpha\) as follows:
			\[
				T_0=\frac{32dL}{\mu},
				\qquad
				\alpha=\frac{1}{\sqrt{d(T+T_0)}}.
			\]
			\FOR {$t=0,1,2,\dots, T-1$ }
			\STATE Draw $\bu_t \sim \mathcal N(0, \bm{I}_d)$ and, independently, a fresh oracle sample \(\xi_t\).
			\STATE Set the stepsize
			\[
				\eta_t=\frac{4d}{\mu(t+T_0)\norm{\bu_t}^2}.
			\]
			\STATE Evaluate \(f(\bx_t +\alpha \bu_t;\xi_t)\) and \(f(\bx_t -\alpha \bu_t;\xi_t)\), and set
			\begin{equation*}
				\bg(\bx_t;\xi_t) = \frac{f(\bx_t+\alpha\bu_t;\xi_t) - f(\bx_t-\alpha\bu_t;\xi_t)}{2\alpha} \bu_t.
			\end{equation*}
			\STATE Update
			\begin{equation}
				\bx_{t+1} = \bx_t - \eta_t\cdot \bg(\bx_t;\xi_t).
			\end{equation}
			\ENDFOR
			\STATE {\bf Output:} $\bx_T$.
		\end{algorithmic}
	\end{small}
\end{algorithm}

The stepsize is chosen after sampling \(\bu_t\) and before applying the
update. Here \(L\) and \(\mu\) are the smoothness and strong-convexity
parameters in Assumption~\ref{ass:Lmu}. This rule is the norm-normalized
analogue of the classical first-order schedule
\(\Theta(1/(\mu(t+T_0)))\) \citep{liu2023revisiting}. The normalization makes
\(\eta_t\|\bu_t\|^2=4d/[\mu(t+T_0)]\) deterministic in the descent recursion,
whereas the factor \(d\) compensates for the fact that one isotropic direction
captures only a \(1/d\) fraction of the squared gradient norm in expectation.
Because \(\|\bu_t\|^2=\mathcal O(d)\) with high probability, the realized stepsize
has the usual stochastic-approximation order. On the null event
\(\|\bu_t\|=0\), we set \(\eta_t=0\) and leave the iterate unchanged; this
convention has no effect under a Gaussian direction law.
Throughout the paper the initial point \(\bx_0\) is deterministic. A random
initialization can be handled by conditioning on it, or by replacing
\(\Delta_0\) below with a deterministic upper bound.

\begin{proposition}[Almost-sure nonvanishing gradients]
	\label{prop:nonvanishing_gradients}
	Suppose that Assumption~\ref{ass:Lmu} holds, \(d\ge2\), and
	\(\bx_0\ne\bx^*\). For every fixed finite horizon \(T\), the iterates
	generated by Algorithm~\ref{alg:SA_1} satisfy
	\[
	\mathbb P\left(
	\nabla f(\bx_t)\ne0\ \text{for all }t=0,\dots,T
	\right)
	=1.
	\]
\end{proposition}

\begin{proof}
	Because \(f\) is differentiable and strongly convex, \(\bx^*\) is the
	unique stationary point of \(f\). It is therefore enough to prove that
	\(\bx_t\ne\bx^*\) for all \(t\le T\) almost surely.
	
	The claim holds at \(t=0\) by assumption. Conditional on the history
	immediately before sampling \(\bu_t\), suppose that
	\(\bx_t\ne\bx^*\). The update displacement in
	Algorithm~\ref{alg:SA_1} is a scalar multiple of the Gaussian direction
	\(\bu_t\). Hence the identity \(\bx_{t+1}=\bx^*\) would require
	\(\bu_t\in\operatorname{span}\{\bx_t-\bx^*\}\), independently of the
	realized scalar multiplying \(\bu_t\). This is a one-dimensional
	subspace, and the conditional Gaussian law of \(\bu_t\) assigns it
	probability zero when \(d\ge2\). A finite induction over
	\(t=0,\dots,T-1\) proves the result.
\end{proof}

By Proposition~\ref{prop:nonvanishing_gradients}, the normalized gradients are
well defined on a probability-one event. On this event, define the angle ratio
\[
\zeta_t
:=
\frac{(\bu_t^\top\nabla f(\bx_t))^2}
{\|\bu_t\|^2\|\nabla f(\bx_t)\|^2}.
\]
This definition does not stop or modify the algorithm. The normalized vector
\(\nabla f(\bx_t)/\|\nabla f(\bx_t)\|\) is measurable before \(\bu_t\) is
sampled. Rotational invariance therefore gives
\(\zeta_t\mid\mathcal F_t^{-}\sim\mathrm{Beta}(1/2,(d-1)/2)\). This
conditional beta law is the sole distributional input to the angle
concentration arguments.
\subsection{Assumptions}

We impose standard smoothness and strong-convexity assumptions.
\begin{assumption}
	\label{ass:Lmu}
	The function \(f\) is \(L\)-smooth and \(\mu\)-strongly convex. That is,
	for every \(\bx,\by\in\RR^d\),
	\begin{align}
		f(\by) \leq f(\bx) + \dotprod{\nabla f(\bx), \by - \bx} + \frac{L}{2}\norm{\by - \bx}^2, \label{eq:L}\\
		f(\by) \geq f(\bx) + \dotprod{\nabla f(\bx), \by - \bx} + \frac{\mu}{2}\norm{\by - \bx}^2.
	\end{align}
\end{assumption}

Throughout, we take \(L\ge\mu\) without loss of generality: replacing \(L\)
by \(L\vee\mu\) preserves both deterministic and stochastic smoothness.

For the stochastic representation in Eq.~\eqref{eq:f_sto}, we impose the
following oracle assumptions.
\begin{assumption}
	\label{ass:est}
	For almost every \(\xi\), the stochastic objective \(f(\cdot;\xi)\) is
	\(L\)-smooth, and \((\bx,\xi)\mapsto\nabla f(\bx;\xi)\) is jointly
	measurable. We also assume that the stochastic gradient
	\(\nabla f(\bx;\xi)\) is an unbiased estimator of \(\nabla f(\bx)\):
	\[
	\EE_\xi[\nabla f(\bx;\xi)] = \nabla f(\bx),
	\qquad \bx\in\RR^d.
	\]
\end{assumption}

\begin{assumption}[Conditional sub-Gaussian stochastic-gradient noise]
	\label{ass:SubG}
	Let
	\[
	\be(\bx;\xi):=\nabla f(\bx)-\nabla f(\bx;\xi).
	\]
	Following Assumption~5B of \citet{liu2023revisiting}, suppose that the
	squared stochastic-gradient error norm satisfies a conditional
	exponential-moment bound: there exists \(\sigma>0\) such that, for every
	\(\bx\) and every \(\lambda\in(0,\sigma^{-2})\),
	\begin{equation}
		\label{eq:subg_noise_assumption}
		\EE_\xi\left[
		\exp\left(\lambda\|\be(\bx;\xi)\|^2\right)
		\right]
		\le
		\exp(\lambda\sigma^2).
	\end{equation}
\end{assumption}

\begin{proposition}[Conditional form along Algorithm~\ref{alg:SA_1}]
\label{prop:algorithm_conditional_subg}
	Suppose that Assumptions~\ref{ass:est} and~\ref{ass:SubG} hold, and
	consider Algorithm~\ref{alg:SA_1}. All independence statements below follow
	from the algorithm's sampling rule; no additional oracle assumption is
	imposed. Conditional on the past, \(\bu_t\) is independent of \(\xi_t\),
	and all future Gaussian directions and oracle samples are fresh relative to
	the current history. Consequently, the exponential-moment bound in
	Eq.~\eqref{eq:subg_noise_assumption} holds conditionally along the
	iterates.
	Let
	\(\mathcal F_t^{-}\) denote the history immediately before sampling
	\(\bu_t\), and let
	\(\mathcal F_t^{u}:=\mathcal F_t^{-}\vee\sigma(\bu_t)\) denote the history
	after sampling \(\bu_t\) but before sampling \(\xi_t\). Let
	\(\mathcal F_t^{+}\) denote the history after sampling \(\xi_t\) and
	completing the update to \(\bx_{t+1}\). Whenever the indices are in range,
	these histories satisfy
	\[
	\mathcal F_t^{-}
	\subseteq
	\mathcal F_t^{u}
	\subseteq
	\mathcal F_t^{+}
	\subseteq
	\mathcal F_{t+1}^{-}
	\subseteq
	\mathcal F_{t+1}^{u}.
	\]
	If
	\(\be_t:=\nabla f(\bx_t)-\nabla f(\bx_t;\xi_t)\), then, for either choice
	\(\mathcal G_t\in\{\mathcal F_t^{-},\mathcal F_t^{u}\}\),
	\[
		\EE[\be_t\mid\mathcal G_t]=0,
		\qquad
		\EE\left[
		\exp\left(\lambda\|\be_t\|^2\right)
		\middle|\mathcal G_t
		\right]
		\le
		\exp(\lambda\sigma^2),
		\qquad 0<\lambda<\sigma^{-2}.
	\]
	The pre-direction form is used when the Gaussian direction is integrated
	out in the quadratic-noise estimate, whereas the post-direction form is used
	to establish conditional sub-Gaussian projection bounds.
\end{proposition}

\begin{proof}
	The iterate \(\bx_t\) is \(\mathcal F_t^{-}\)-measurable, and
	Algorithm~\ref{alg:SA_1} samples \(\bu_t\) and a fresh oracle sample
	\(\xi_t\) independently at iteration \(t\). Hence conditioning on either
	\(\mathcal F_t^{-}\) or \(\mathcal F_t^{u}\) fixes \(\bx_t\) (and, in the
	latter case, \(\bu_t\)) while leaving \(\xi_t\) distributed as a fresh
	oracle sample. Assumption~\ref{ass:est} gives the conditional mean-zero
	identity, and Assumption~\ref{ass:SubG}, applied pointwise at the realized
	\(\bx_t\), gives the displayed conditional exponential-moment bound.
\end{proof}

Assumption~\ref{ass:est} is standard. Assumption~\ref{ass:SubG} is the
full-vector conditional exponential-moment condition used by
\citet{liu2023revisiting}; it is weaker than almost-sure bounded noise but
stronger than coordinatewise tail control. Proposition~\ref{prop:algorithm_conditional_subg}
transfers this oracle condition to the filtrations generated by
Algorithm~\ref{alg:SA_1}. We do not truncate the noise or condition on a
bounded-noise event. Instead, a conditional sub-Gaussian projection inequality
controls the signed linear term, and the exponential moment of
\(\|\be_t\|^2\) controls the quadratic term. Constants such as
\(C_{\mathrm{sg}}\), \(C_Y\), and \(C_{\mathrm{abs}}\) denote universal
numerical factors and do not alter the displayed rate or its confidence
scaling.

Assumption~\ref{ass:SubG} includes bounded stochastic-gradient noise and also
allows unbounded light-tailed models satisfying the stated exponential-moment
condition. Examples include stochastic quadratic objectives with random linear
terms, finite-support simulation models, and empirical-risk objectives with the
same full-vector tail control. Using the same sample in both function
evaluations is the common-random-numbers construction familiar in simulation
optimization. It removes the additional differencing noise generated by
independent samples and leaves \(\be(\bx;\xi)\) as the stochastic error to be
controlled.

When conditional Orlicz notation is used, for a sigma-field \(\mathcal H\) and
\(p\in\{1,2\}\) we write
\[
	\|X\|_{\psi_p\mid\mathcal H}
	:=
	\inf\left\{
	a>0:
	\EE\left[
	\exp\left(\frac{|X|^p}{a^p}\right)
	\middle|\mathcal H
	\right]\le2\ \text{a.s.}
	\right\},
\]
with the usual convention that the infimum is \(+\infty\) if the set is empty.
Changing the numerical constant \(2\) only changes universal constants.

\section{Main Results}

We state the finite-horizon last-iterate guarantee for
Algorithm~\ref{alg:SA_1}. The finite-horizon theorem keeps the product-weight
factor \(c_\rho=\exp\{O(\Gamma_T(\delta)/T_0)\}\) explicit. In particular, when
\(T_0\gtrsim\Gamma_T(\delta)\), the norm-normalized two-point Gaussian method
satisfies
\[
	f(\bx_T)-f(\bx^*)
	=
	\widetilde{\mathcal O}\!\left(\frac{d}{T}\right)
\]
with probability at least \(1-\delta\), after fixed problem parameters and
logarithmic factors are suppressed. This is the high-probability last-iterate
analogue of the classical expected zeroth-order rate.

Throughout this section, Algorithm~\ref{alg:SA_1} is run with the displayed
choices of \(T_0\), \(\alpha\), and \(\eta_t\). These parameters satisfy the
stability condition \(\eta_t\le1/(8L\norm{\bu_t}^2)\). For compactness, define
\[
	T_k:=k+T_0,
	\qquad
	\Lambda:=\log(T+T_0),
	\qquad
	J_T
	:=
	1+
	\left\lceil
	\log_2
	\left(
	\frac{T(T+T_0)}{T_0}
	\right)
	\right\rceil,
\]
write
\[
	\Delta_0:=f(\bx_0)-f(\bx^*),
\]
and define the stitched confidence factor
\[
	\Gamma_T(\delta)
	:=
	1+\log\frac1\delta+\log J_T+\log(e+T_0),
	\qquad
	c_\delta:=\frac{\Gamma_T(\delta)}{\Lambda},
	\qquad
	\Theta_T:=\Lambda\Gamma_T(\delta).
\]
Let \(C>0\) be the universal constant in
Lemma~\ref{lem:weighted_scan_bounds}, and set
\[
	c_\rho
	:=
	2\exp\left(
	\frac{2C}{T_0}\Gamma_T(\delta)
	\right).
\]
Here \(c_\rho\) is the suffix constant supplied by the weighted scan,
\(c_\delta\Lambda=\Gamma_T(\delta)\) collects the confidence logarithms, and
\(\Theta_T\) is the induction scale.

\begin{theorem}[High-probability last-iterate rate]
	\label{thm:main_1}
	Let \(T\ge1\), \(0<\delta<1\), \(d\ge2\), and let the deterministic initial
	point satisfy \(\bx_0\ne\bx^*\). Suppose that \(f\) satisfies
	Assumption~\ref{ass:Lmu}, that the stochastic objectives \(f(\cdot;\xi)\)
	satisfy Assumption~\ref{ass:est}, and that the conditional sub-Gaussian
	noise condition in Assumption~\ref{ass:SubG} holds. Run
	Algorithm~\ref{alg:SA_1} for \(T\) iterations, using the same sample
	\(\xi_t\) in both function evaluations at iteration \(t\).
	Let \(C_{\mathrm{abs}}\ge1\) be a universal numerical constant chosen
	sufficiently large relative to the constants in the auxiliary concentration
	inequalities and Lemma~\ref{lem:time_uniform_controls}. Define
	\begin{equation}\label{eq:cC}
		\begin{aligned}
			\cC := \max\Bigg\{
			&\frac{8c_\rho T_0\Delta_0}{d\Theta_T},
			\frac{C_{\mathrm{abs}}L\sigma^2c_\rho^2}{\mu^2}
			\left(1+\frac{c_\delta}{T_0}\right),
			\frac{C_{\mathrm{abs}}Lc_\rho\sigma^2}{\mu^2\Gamma_T(\delta)}
			\left(1+\frac{dc_\delta}{T_0}\right),
			\\
			&
			\frac{C_{\mathrm{abs}}c_\rho(1+c_\delta)}{\Gamma_T(\delta)}
			\left(
			1+\frac{L^2}{\mu}+\frac{L^3}{\mu^2}
			\right)
			\Bigg\}.
		\end{aligned}
	\end{equation}
	For \(0\le K\le T\), define
	\begin{equation}
		\cE_{\Delta_K}
		:=
		\left\{
		f(\bx_k)-f(\bx^*)
		\le
		\frac{\cC d\Theta_T}{k+T_0}
		\quad\text{for all } k=0,\dots,K
		\right\}.
	\end{equation}
	Then, with probability at least \(1-\delta\), the events
	\(\cE_{\Delta_K}\) hold simultaneously for all \(K\le T\). In particular,
	\begin{equation}
		\label{eq:improved_rate_weighted}
		f(\bx_T)-f(\bx^*)
		\le
		\cC\,
		\frac{d\Theta_T}{T+T_0}
		=
		\mathcal O\left(
		\frac{c_\rho T_0\Delta_0}{T+T_0}
		+
		c_\rho^2
		\frac{d\Lambda\Gamma_T(\delta)}{T+T_0}
		\left(1+\frac{\Gamma_T(\delta)}{T_0}\right)
		\right).
	\end{equation}
\end{theorem}

Theorem~\ref{thm:main_1} yields the following fixed-horizon convergence rate.

\begin{corollary}[Simplified high-probability convergence rate]
	\label{cor:main_1}
	Let the assumptions and parameter choices of Theorem~\ref{thm:main_1}
	hold. If \(T_0\gtrsim\Gamma_T(\delta)\), so that \(c_\rho\) is bounded by
	a universal numerical constant, then with probability at least
	\(1-\delta\),
	\begin{equation}
	\label{eq:simplified_rate}
		f(\bx_T)-f(\bx^*)
		\le
		C
		\left[
		\frac{T_0\Delta_0}{T+T_0}
		+
		\frac{d\,\Lambda\Gamma_T(\delta)}{T+T_0}
		\left(1+\frac{\Gamma_T(\delta)}{T_0}\right)
		\right],
	\end{equation}
	where \(C\) depends only on universal numerical constants and the fixed
	parameters in Eq.~\eqref{eq:cC}. Suppressing these parameters and logarithmic
	factors, the last iterate satisfies
	\[
		f(\bx_T)-f(\bx^*)
		=
		\widetilde{\mathcal O}\!\left(\frac{d}{T}\right).
	\]
\end{corollary}

\begin{corollary}[High-probability oracle complexity for admissible horizons]
	\label{cor:oracle_complexity}
	Under the assumptions and parameter choices of
	Corollary~\ref{cor:main_1}, let \(\varepsilon>0\). Suppose that an admissible
	integer horizon \(T\) exists satisfying the clean-rate compatibility
	condition
	\[
		\Gamma_T(\delta)\lesssim T_0=\frac{32dL}{\mu}
	\]
	and
	\[
		T+T_0
		\ge
		\frac{C}{\varepsilon}
		\left[
		T_0\Delta_0
		+
		d\,\Lambda\Gamma_T(\delta)
		\left(1+\frac{\Gamma_T(\delta)}{T_0}\right)
		\right].
	\]
	Then \(f(\bx_T)-f(\bx^*)\le\varepsilon\) with probability at least
	\(1-\delta\). Thus, for accuracy levels admitting such a horizon, the
	iteration complexity is \(\widetilde{\mathcal O}(d/\varepsilon)\) after
	fixed problem parameters and logarithmic factors are suppressed. This is
	also the number of two-point oracle calls; the number of individual function
	evaluations is twice as large.
\end{corollary}

\begin{proof}
	The displayed lower bound on \(T+T_0\) makes the right-hand side of
	Eq.~\eqref{eq:simplified_rate} no larger than \(\varepsilon\). The
	compatibility condition \(\Gamma_T(\delta)\lesssim T_0\) ensures that
	Corollary~\ref{cor:main_1} applies at the selected horizon, so its
	probability guarantee carries over directly.
	Because \(T_0=\Theta(dL/\mu)\) and the remaining dependence on \(T\) and
	\(\delta\) enters logarithmically through \(\Lambda\) and
	\(\Gamma_T(\delta)\), the bound yields
	\(\widetilde{\mathcal O}(d/\varepsilon)\) iterations whenever an admissible
	horizon exists, after fixed problem parameters and logarithmic factors are
	suppressed. Each iteration makes one two-point oracle
	query, that is, two function evaluations using the same sample.
\end{proof}

\begin{remark}[Interpretation of the corollaries]
	Corollary~\ref{cor:main_1} is a fixed-horizon result, whereas
	Corollary~\ref{cor:oracle_complexity} expresses the same inequality as an
	\(\varepsilon\)-accuracy statement. The confidence level enters only through
	\(\Gamma_T(\delta)\), in contrast to the polynomial \(1/\delta\) loss from
	a Markov conversion. Theorem~\ref{thm:main_1} itself keeps the product-weight
	factor \(c_\rho=\exp\{O(\Gamma_T(\delta)/T_0)\}\) explicit. The simplifying
	condition \(\Gamma_T(\delta)\lesssim T_0=32dL/\mu\) is used only to keep this
	factor universal in Corollary~\ref{cor:main_1}; the oracle-complexity
	corollary is stated for horizons satisfying the same clean-rate condition.
	Consequently, the \(\widetilde{\mathcal O}(d/\varepsilon)\) interpretation is
	conditional, not a uniform fixed-\(d\) complexity statement as
	\(\varepsilon\downarrow0\). For fixed \(d\) and \(\delta\), the admissible
	range is finite because \(\Gamma_T(\delta)\) grows logarithmically with
	\(T\), although the range can be extremely large when \(T_0=32dL/\mu\) is
	large.
	Finally, the notation \(\widetilde{\mathcal O}(d/T)\) treats the full-vector
	noise scale \(\sigma\) as a fixed problem parameter. Any dimension dependence
	of \(\sigma\) remains explicit through the constants in
	Theorem~\ref{thm:main_1}.
\end{remark}

\section{Convergence Analysis of Zeroth-Order Stochastic Gradient Descent}
\label{sec:zsgd}

This section proves Theorem~\ref{thm:main_1}. We first derive a one-step
descent inequality, then unroll it into weighted suffix terms and control each
term on a common high-probability event.

\begin{lemma}
\label{lem:dec_1}
Let \(f\) be \(L\)-smooth and \(\mu\)-strongly convex, and let
\(f(\cdot;\xi)\) be \(L\)-smooth. Assume \(\nabla f(\bx_t)\ne0\). If
\[
	\eta_t\le \frac{1}{8L\|\bu_t\|^2},
\]
then
\begin{equation}
	\label{eq:dec_1}
	\begin{aligned}
		\Delta_{t+1}
		\leq&
		\left(1 - \frac{\mu\eta_t}{2}
		\frac{(\bu_t^\top\nabla f(\bx_t))^2}{\|\nabla f(\bx_t)\|^2}
		\right)\cdot \Delta_t
		+
	\eta_t \dotprod{\nabla f(\bx_t), \bu_t\bu_t^\top \be_t}
	\\
	&
	+
	2L\eta_t^2\norm{\bu_t}^2|\bu_t^\top \be_t|^2
	+ \Delta_{\alpha, t}
	\end{aligned}
\end{equation}
where \(\Delta_t\), \(\be_t\), and \(\Delta_{\alpha,t}\) are defined by
\begin{equation}
\label{eq:def}
\Delta_t := f(\bx_t)-f(\bx^*),
\qquad
\be_t := \nabla f(\bx_t)-\nabla f(\bx_t;\xi_t),
\qquad
\Delta_{\alpha,t}
:=
\frac{\eta_t\beta_t^2}{2}
+L\eta_t^2\beta_t^2\norm{\bu_t}^2.
\end{equation}
Here \(\beta_t\) is the finite-difference bias coefficient defined by
\[
\bg(\bx_t;\xi_t)
=
\bu_t\bu_t^\top\nabla f(\bx_t;\xi_t)+\beta_t\bu_t,
\qquad
|\beta_t|\le \frac{L\alpha}{2}\|\bu_t\|^2.
\]
\end{lemma}

Lemma~\ref{lem:dec_1} is the basic one-step estimate. Iterating
Eq.~\eqref{eq:dec_1} produces a suffix-product recursion. The following two
lemmas control its random contraction factors and state the resulting unrolled
inequality. In the main theorem, the nonzero-gradient condition in
Lemma~\ref{lem:dec_1} holds almost surely by
Proposition~\ref{prop:nonvanishing_gradients}.

\begin{lemma}
\label{lem:cE_rho}
Suppose that Assumption~\ref{ass:Lmu} holds, \(d\ge2\), and
\(\bx_0\ne\bx^*\). Let \(\{\bu_t\}\) and \(\{\bx_t\}\) be the sequences
generated by Algorithm~\ref{alg:SA_1} with fresh independent Gaussian
directions. Let \(T_0=32dL/\mu\), \(T_t=t+T_0\),
\(\Lambda:=\log(T+T_0)\), and \(0<\delta<1\). Define
\[
	J_T
	:=
	1+
	\left\lceil
	\log_2
	\left(
	\frac{T(T+T_0)}{T_0}
	\right)
	\right\rceil
\]
and let \(\Gamma_T(\delta)\) be any number satisfying
\[
	\Gamma_T(\delta)\ge
	1+\log\frac1\delta+\log J_T+\log(e+T_0).
\]
Let \(C>0\) be a universal constant large enough for
Lemma~\ref{lem:weighted_scan_bounds}. Define the events
\begin{align}
	\cE_{\rho_K} :=& \left\{ \exp\left(- 2d \sum_{t=0}^{K-1} \frac{\zeta_t}{T_t} \right) \leq \rho_K \right\}
	\\
	\cE_{\rho_{K,k}} :=& \left\{ \exp\left(- 2d \sum_{t=k+1}^{K-1} \frac{\zeta_t}{T_t} \right) \leq \rho_{K,k} \right\},
\end{align}
and the prefix upper-tail event
\[
	\cE_{\rho}^{+}
	:=
	\left\{
	d\sum_{t=0}^{K-1}\frac{\zeta_t}{T_t}
	\le
	C\Lambda+C\frac{\Gamma_T(\delta)}{T_0}
	\quad K=1,\dots,T
	\right\}.
\]
Here \(\zeta_t\) is the angle variable defined after
Proposition~\ref{prop:nonvanishing_gradients}; it is well-defined almost
surely under the assumptions above.
The constants \(\rho_K\) and \(\rho_{K,k}\) are defined by
\begin{align}
	\rho_K :=& \exp\left(- \sum_{t=0}^{K-1} \frac{1}{T_t} + \frac{2C}{T_0}\Gamma_T(\delta) \right) , \label{eq:rho_k}
	\\
	\rho_{K,k}:=& \exp\left(- \sum_{t=k+1}^{K-1} \frac{1}{T_t} + \frac{2C}{T_0}\Gamma_T(\delta) \right). \label{eq:rho_kk}
\end{align}
Then, with probability at least \(1-\delta\), the joint event
\[
	\bigcap_{K=1}^{T}
	\left[
	\cE_{\rho_K}\cap\left(\cap_{k=0}^{K-1}\cE_{\rho_{K,k}}\right)
	\right]
	\cap \cE_{\rho}^{+}
\]
holds. In particular, the suffix-product comparisons used in
Lemma~\ref{lem:dec_2} hold simultaneously for every terminal time \(K\le T\).
\end{lemma}

Lemma~\ref{lem:cE_rho} lower-bounds the weighted angle sums induced by the
schedule \(\eta_t=\Theta(1/(t+T_0))\)
\citep{bottou2018optimization,liu2023revisiting}. These bounds yield the
following unrolled recursion for an arbitrary terminal time \(K\).

\begin{lemma}
\label{lem:dec_2}
Suppose that the hypotheses of Lemma~\ref{lem:dec_1} hold along the trajectory
of Algorithm~\ref{alg:SA_1}. Set
\(\eta_t=4d/[\mu(t+T_0)\norm{\bu_t}^2]\) with
\(T_0=32dL/\mu\), and write \(T_t:=t+T_0\). Suppose that
\(\cE_{\rho_K}\) and \(\cE_{\rho_{K,k}}\), \(k=0,\dots,K-1\), hold.
Define the actual suffix product
\[
	\Pi_{K,k}
	:=
	\prod_{t=k+1}^{K-1}
	\left(
	1-\frac{2d}{T_t}\zeta_t
	\right),
	\qquad k=0,\dots,K-1,
\]
with the empty product equal to \(1\).
Then
\begin{equation}\label{eq:dec_2}
\begin{aligned}
\Delta_K
\leq&
\rho_K \cdot \Delta_0
+ \sum_{k=0}^{K-1} \Pi_{K,k} \left( \frac{4d \cdot \dotprod{\nabla f(\bx_k), \bu_k\bu_k^\top \be_k}}{\mu (k+T_0) \norm{\bu_k}^2} \right)
\\
&+
\sum_{k=0}^{K-1} \rho_{K,k} \left(\frac{32Ld^2 }{\mu^2 (k+T_0)^2} \cdot \frac{|\bu_k^\top \be_k|^2}{\norm{\bu_k}^2}
\right)
\\
&+
\sum_{k=0}^{K-1} \rho_{K,k} \left( \frac{ d L^2\alpha^2 \norm{\bu_k}^2}{2\mu(k+T_0)}
+
\frac{ 4 d^2 L^3\alpha^2 \norm{\bu_k}^2}{ \mu^2 (k+T_0)^2 } \right).
\end{aligned}
\end{equation}
\end{lemma}

The following subsections control the terms on the right-hand side of
Eq.~\eqref{eq:dec_2}.

\subsection{Two Consequences of Conditional Sub-Gaussian Noise}

We require two consequences of Assumption~\ref{ass:SubG}, corresponding to the
linear and quadratic stochastic-error terms. They replace the uses of an
almost-sure noise bound in a bounded-noise analysis; in particular, no event of
the form \(\max_t\|\be_t\|\le R\) is introduced.

\begin{lemma}[Conditional sub-Gaussian projection]
\label{lem:subg_projection}
	Let \(\mathcal F\) be a \(\sigma\)-algebra. Suppose that a random vector
	\(\be\in\mathbb R^d\) satisfies
	\[
		\EE[\be\mid\mathcal F]=0,
		\qquad
		\EE\left[
		\exp\left(\lambda\|\be\|^2\right)
		\middle|\mathcal F
		\right]
		\le
		\exp(\lambda\sigma^2),
		\qquad
		0<\lambda<\sigma^{-2}.
	\]
	Then there is a universal numerical constant \(C_{\mathrm{sg}}\ge1\) such
	that, for every \(\mathcal F\)-measurable vector \(\bz\in\mathbb R^d\) and
	every \(\theta\in\mathbb R\),
	\begin{equation}
	\label{eq:subg_projection_mgf}
		\EE\left[
		\exp\left(\theta\langle \be,\bz\rangle\right)
		\middle|\mathcal F
		\right]
		\le
		\exp\left(C_{\mathrm{sg}}\theta^2\sigma^2\|\bz\|^2\right).
	\end{equation}
\end{lemma}

Lemma~\ref{lem:subg_projection} provides the scalar conditional
sub-Gaussian estimate needed to control the signed linear stochastic error by a
product-martingale argument.

\begin{lemma}[Random-direction quadratic sub-exponential bound]
\label{lem:subg_random_direction_square}
	Let \(\mathcal F\) be a \(\sigma\)-algebra, and suppose that
	\(\be\in\mathbb R^d\) satisfies the conditional mean-zero and
	exponential-moment conditions in Lemma~\ref{lem:subg_projection}. Let
	\(\bu\sim\mathcal N(0,I_d)\) be independent of \(\mathcal F\vee\sigma(\be)\). In applications,
	\(\mathcal F\) is the pre-direction \(\sigma\)-field, so conditioning is
	performed before \(\bu\) is drawn and the Gaussian direction can be
	integrated out. Then, for every
	\(0<\lambda\le 1/(2\sigma^2)\),
	\begin{equation}
	\label{eq:subg_random_direction_square_mgf}
		\EE\left[
		\exp\left(
			\lambda
			\frac{(\bu^\top\be)^2}{\|\bu\|^2}
		\right)
		\middle|\mathcal F
		\right]
		\le
		\exp\left(\frac{2\lambda\sigma^2}{d}\right).
	\end{equation}
	Consequently, the following sequential version holds. Suppose that, at
	times \(k=0,\dots,K-1\), \(\be_k\) satisfies the same conditional
	mean-zero and exponential-moment bounds with respect to the pre-direction
	history, and
	\(\bu_k\sim\mathcal N(0,I_d)\) is conditionally independent of that
	history and of \(\be_k\). Then, for any nonnegative deterministic weights
	\(w_0,\dots,w_{K-1}\), with \(W=\sum_{k=0}^{K-1}w_k\) and
	\(w_{\max}=\max_{0\le k\le K-1}w_k\), the following inequality holds
	with probability at least \(1-\delta\):
	\begin{equation}
	\label{eq:subg_weighted_quadratic}
		\sum_{k=0}^{K-1}
		w_k\frac{(\bu_k^\top\be_k)^2}{\|\bu_k\|^2}
		\le
		\frac{2\sigma^2}{d}W
		+
		2\sigma^2 w_{\max}\log\frac1\delta.
	\end{equation}
\end{lemma}

Lemma~\ref{lem:subg_random_direction_square} is used for the quadratic-noise
term in Eq.~\eqref{eq:dec_2}; it integrates over the Gaussian direction
directly and avoids an almost-sure bound on \(\|\be_k\|\).

\subsection{Linear Martingale Bound}

We next control the signed suffix-product linear term. Crucially, the realized
product \(\Pi_{K,k}\) is retained through the martingale concentration step;
the deterministic comparison \(\Pi_{K,k}\le\rho_{K,k}\) is used only when
bounding the resulting variance proxy. For a fixed terminal time \(K\), define
\begin{equation}
\label{eq:T_Y}
T_k:=k+T_0,
\qquad
Y_{K,k}^{\Pi}
:=
\frac{\Pi_{K,k}}{T_k}
\cdot
\frac{\nabla f(\bx_k)^\top \bu_k \, \bu_k^\top \be_k}{\|\bu_k\|^2},
\qquad k=0,\dots,K-1,
\end{equation}
where \(\be_k\) is defined in Eq.~\eqref{eq:def}.

\begin{lemma}[Future angle innovations are conditionally independent of current noise]
\label{lem:future_angle_innovations}
	Fix \(K\) and \(k<K\). Let \(\mathcal F_k^+\) be the post-update history
	defined in Proposition~\ref{prop:algorithm_conditional_subg}.
	Conditional on \(\mathcal F_k^+\), the future angle variables
	\[
		\zeta_t
		=
		\frac{(\bu_t^\top\nabla f(\bx_t))^2}
		{\|\bu_t\|^2\|\nabla f(\bx_t)\|^2},
		\qquad t=k+1,\dots,K-1,
	\]
	have the product law
	\[
		\nu^{\otimes(K-k-1)},\qquad
		\nu=\mathrm{Beta}\left(\frac12,\frac{d-1}{2}\right),
	\]
	where \(\nu^{\otimes(K-k-1)}\) denotes the \((K-k-1)\)-fold product
	measure, equivalently, the law of \(K-k-1\) conditionally independent
	random variables with common distribution \(\nu\).
	Although the future iterates depend on the realized value of \(\xi_k\)
	through \(\bx_{k+1}\), this conditional angle law is deterministic. More
	precisely, with \(\mathcal F_k^{u}\) as defined in
	Proposition~\ref{prop:algorithm_conditional_subg},
	\(\sigma(\zeta_{k+1},\dots,\zeta_{K-1})\) is conditionally independent of
	\(\xi_k\), and hence of \(\be_k\), given \(\mathcal F_k^{u}\). Therefore,
	before sampling \(\xi_k\), one may enlarge the conditioning
	\(\sigma\)-field by
	\(\zeta_{k+1},\dots,\zeta_{K-1}\) without changing the conditional
	distributional bounds for \(\be_k\).
\end{lemma}

Lemma~\ref{lem:future_angle_innovations} justifies the enlarged filtration in
the product-martingale argument: future angle variables may be revealed
without changing the current noise law.

\paragraph{Product-linear event and variance proxy.}
For \(1\le K\le T\), let \(Y_{K,k}^{\Pi}\) be defined in
Eq.~\eqref{eq:T_Y}. Given a deterministic envelope \(\bar A_K>0\) and an
auxiliary confidence parameter \(\delta_K^{(Y)}\in(0,1)\), define
\begin{equation}
\cE_{Y_K}^{\Pi}(\bar A_K)
:=
\left\{
	\sum_{k=0}^{K-1} Y_{K,k}^{\Pi}
	\le
	C_Y\sigma
	\sqrt{\bar A_K\log\frac{1}{\delta_K^{(Y)}}}
\right\},
\end{equation}
where \(C_Y\ge1\) is a universal numerical constant fixed large enough in the
time-uniform product-martingale argument below. Also define
\begin{equation}
	A_K^{\Pi}
	:=
	\sum_{k=0}^{K-1}
	\frac{\Pi_{K,k}^2}{T_k^2}
	\frac{(\nabla f(\bx_k)^\top \bu_k)^2}
	{\|\bu_k\|^2\|\nabla f(\bx_k)\|^2}
	\|\nabla f(\bx_k)\|^2.
\end{equation}
On the probability-one event of Proposition~\ref{prop:nonvanishing_gradients},
the normalized factor is precisely the angle ratio \(\zeta_k\), so this is the
same intrinsic variance proxy as
\(\sum_k\Pi_{K,k}^2T_k^{-2}(\nabla f(\bx_k)^\top\bu_k)^2/\|\bu_k\|^2\).
In the application to the induction proof, we compare this variance proxy with
the deterministic envelope
\begin{equation}
\label{eq:Abar_product_linear}
	\bar A_K
	:=
	C_A
	\frac{L\cC c_\rho^2\Theta_T\Lambda}{T_K^2}
	\left(1+\frac{c_\delta}{T_0}\right),
	\qquad 1\le K\le T,
\end{equation}
where \(C_A\) is a sufficiently large universal numerical constant. The
product-linear lemma below gives the time-uniform high-probability bound for
the signed product-linear term and is stated for arbitrary deterministic
envelopes; Eq.~\eqref{eq:Abar_product_linear} is the choice used when the
bound is assembled into the time-uniform induction event. The lemma is
conditional on the product-envelope event \(\mathcal P_T(C_P)\);
Lemma~\ref{lem:time_uniform_controls} verifies this envelope from the weighted
scan event.

\begin{lemma}[Time-uniform signed product-linear bound]
\label{lem:product_linear_uniform}
	Let the assumptions of Theorem~\ref{thm:main_1} hold. Define
	\[
		q_t:=1-\frac{2d}{T_t}\zeta_t,\qquad
		P_0:=1,\qquad
		P_K:=\prod_{t=0}^{K-1}q_t,
	\]
	and, for a numerical constant \(C_P\), let
	\[
		\mathcal P_T(C_P)
		:=
		\left\{
		P_K^{-2}\le
		(T+T_0)^{C_P}
		\exp\!\left(C_P\frac{\Gamma_T(\delta)}{T_0}\right),
		\quad K=1,\dots,T
		\right\}.
	\]
	Let \(\bar A_K>0\), \(K=1,\dots,T\), be deterministic envelopes and set
	\[
		U_T:=
		(T+T_0)^{C_P}
		\exp\!\left(C_P\frac{\Gamma_T(\delta)}{T_0}\right)
		\max_{K\le T}\bar A_K.
	\]
	Let \(0<\delta_{\rm lin}<1\) and let
	\(G\ge1+\log(1/\delta_{\rm lin})\) satisfy
	\[
		\max_{K\le T}
		\log\left(1+\log\left(e+\frac{U_T}{\bar A_K}\right)\right)
		\le C\,G.
	\]
	Then, for a universal numerical constant \(C_Y\),
	\begin{align}
	\label{eq:product_linear_uniform_bound}
	\Pr\Bigg(
		\mathcal P_T(C_P)
		\cap
		\Bigg\{
		\exists K\le T:\,
		A_K^\Pi\le\bar A_K,\
		\sum_{k=0}^{K-1}Y_{K,k}^{\Pi}
		>
		C_Y\sigma\sqrt{\bar A_K\,G}
		\Bigg\}
	\Bigg)
	\le \delta_{\rm lin}.
	\end{align}
	Equivalently, on any event implying \(\mathcal P_T(C_P)\), outside an
	additional failure probability \(\delta_{\rm lin}\), the implication
	\[
		A_K^\Pi\le\bar A_K
		\quad\Longrightarrow\quad
		\sum_{k=0}^{K-1}Y_{K,k}^{\Pi}
		\le
		C_Y\sigma\sqrt{\bar A_K\,G}
	\]
	holds simultaneously for every \(K\le T\).
\end{lemma}

Lemma~\ref{lem:product_linear_uniform} is the concentration step for the
signed linear term in Eq.~\eqref{eq:dec_2}. It retains the true suffix product
\(\Pi_{K,k}\) until martingale cancellation has been exploited. The variance
condition \(A_K^\Pi\le\bar A_K\) is later verified directly from the normalized
angle ratios \(\zeta_k\), the weighted scan, and smoothness.

\subsection{Quadratic Noise Term}

Lemma~\ref{lem:subg_random_direction_square} controls the quadratic-noise
term in Eq.~\eqref{eq:dec_2} through the scalar ratio
\(
\frac{(\bu_k^\top\be_k)^2}{\|\bu_k\|^2},
\)
without imposing a uniform bound on \(\|\be_k\|\).
Equivalently, the proof uses the decomposition into the squared angle between
\(\bu_k\) and \(\be_k\) times \(\|\be_k\|^2\), integrating over the Gaussian
direction before applying the full-vector sub-Gaussian moment bound.
\begin{lemma}
\label{lem:Q_up}
Consider Algorithm~\ref{alg:SA_1} under
Assumptions~\ref{ass:est} and~\ref{ass:SubG}, with fresh independent Gaussian
directions and oracle samples. Let \(0<\delta_K^{(Q)}<1\). Define the event
\begin{equation}
\begin{aligned}
\cE_{Q_K}:= \Bigg\{
\sum_{k=0}^{K-1}
\frac{\rho_{K,k}}{T_k^2}
\cdot
\frac{(\bu_k^\top \be_k)^2}{\|\bu_k\|^2}
\leq
\frac{2\sigma^2}{d}
\sum_{k=0}^{K-1} \frac{\rho_{K,k}}{T_k^2}
+
2\sigma^2
\log\frac{1}{\delta_K^{(Q)}}
\max_{0 \leq k \leq K-1}
\frac{\rho_{K,k}}{T_k^2}
\Bigg\}.
\end{aligned}
\end{equation}
The event \(\cE_{Q_K}\) holds with probability at least
\(1-\delta_K^{(Q)}\).
\end{lemma}

Lemma~\ref{lem:Q_up} controls the nonnegative quadratic stochastic-error term
in the unrolled recursion.

\subsection{Smoothing-Bias Term}

It remains to control the smoothing-bias contribution in
Eq.~\eqref{eq:dec_2}.
\begin{lemma}
\label{lem:alp_term_up}
Let \(\{\bu_k\}_{k=0}^{K-1}\) be fresh independent
\(\mathcal N(0,I_d)\) directions. Let \(c_1=dL^2/(2\mu)\),
\(c_2=4d^2L^3/\mu^2\), and \(0<\delta_K^{(\alpha)}<1\). Define the event
\begin{equation}
\begin{aligned}
\cE_{\alpha_K} :=&
\Bigg\{
\sum_{k=0}^{K-1} \rho_{K,k}
\left(
\frac{c_1}{T_k} + \frac{c_2}{T_k^2}
\right)\norm{\bu_k}^2
\leq
C_\chi d
\Bigg[
\sum_{k=0}^{K-1} \rho_{K,k}
\left(
\frac{c_1}{T_k} + \frac{c_2}{T_k^2}
\right)
\\
&
+
\sqrt{ \log\frac{1}{\delta_K^{(\alpha)}} \cdot \sum_{k=0}^{K-1} \rho_{K,k}^2
	\left(
	\frac{c_1}{T_k} + \frac{c_2}{T_k^2}
	\right)^2}
+
	\log\frac{1}{\delta_K^{(\alpha)}} \max_{0 \leq k \leq K-1} \rho_{K,k}
\left(
\frac{c_1}{T_k} + \frac{c_2}{T_k^2}
\right)
\Bigg]
\Bigg\}.
\end{aligned}
\end{equation}
Then \(\cE_{\alpha_K}\) holds with probability at least
\(1-\delta_K^{(\alpha)}\), where \(C_\chi\ge1\) is a universal numerical
constant.
\end{lemma}

Lemma~\ref{lem:alp_term_up} bounds the accumulated bias caused by the
finite-difference radius \(\alpha\).

\subsection{Induction Proof of Theorem~\ref{thm:main_1}}

We conclude the analysis by collecting deterministic bounds on
\(\rho_K\) and \(\rho_{K,k}\), assembling all concentration events, and closing
the induction.
\begin{lemma}[Deterministic bounds on $\rho_K$ and $\rho_{K,k}$]
	\label{lem:c_rho}
	Let \(\rho_K\) and \(\rho_{K,k}\) be defined by
	Eqs.~\eqref{eq:rho_k} and \eqref{eq:rho_kk}, respectively. Then, for every
	\(1\le K\le T\) and \(0\le k\le K-1\),
	\begin{equation}
		\label{eq:c_rho}
		\rho_K
		\le
		c_\rho
		\frac{T_0}{T_K},
		\qquad
		\rho_{K,k}
		\le
		c_\rho
		\frac{T_k}{T_K}.
	\end{equation}
	Here \(C>0\) is the universal constant in
	Lemma~\ref{lem:weighted_scan_bounds} and \(c_\rho\) is defined by
	\begin{equation}
		c_\rho
		:=
		2\exp\!\left(
		\frac{2C}{T_0}\Gamma_T(\delta)
		\right).
	\end{equation}
\end{lemma}

Lemma~\ref{lem:c_rho} converts the exponential \(\rho\)-weights into the
deterministic ratios \(T_0/T_K\) and \(T_k/T_K\), which are used throughout the
final induction.

\begin{lemma}[Time-uniform auxiliary controls]
\label{lem:time_uniform_controls}
Let the assumptions of Theorem~\ref{thm:main_1} hold and let
\[
	\bar\Delta_k:=\frac{\cC d\Theta_T}{T_k},
	\qquad k=0,\dots,T-1.
\]
Set
\[
	\delta_K^{(Y)}
	=
	\delta_K^{(Q)}
	=
	\delta_K^{(\alpha)}
	=:
	\delta_\star,
	\qquad
	\delta_\star:=\exp\{-\Gamma_T(\delta)\},
	\qquad
	\log\frac1{\delta_\star}=c_\delta\Lambda=\Gamma_T(\delta),
\]
and use the deterministic envelopes \(\bar A_K\) in
Eq.~\eqref{eq:Abar_product_linear}.
Here \(C_\chi\) is the universal constant in Lemma~\ref{lem:alp_term_up}. The
quantities \(\delta_K^{(\cdot)}\) specify thresholds in the fixed-\(K\) event
notation; any fixed numerical allocation is absorbed into the universal
constants \(C,c_\rho,\cC,C_A\). Then, with probability at least \(1-\delta\),
the following controls hold:
\begin{enumerate}
	\item the weighted-suffix events
	\(\cE_{\rho_K}\) and \(\cE_{\rho_{K,k}}\), \(k=0,\dots,K-1\), together
	with the prefix upper-tail event \(\cE_\rho^+\), for every \(K\le T\);
	in addition, the prefix estimate
	\[
		d\sum_{t=0}^{K-1}\frac{\zeta_t}{T_t}
		\le
		C\Lambda\left(1+\frac{c_\delta}{T_0}\right),
		\qquad K=1,\dots,T,
	\]
	holds after enlarging the universal constant \(C\);
	\item the product-linear implication event
	\[
		\mathcal L_K(\bar A_K)
		:=
		\cE_{Y_K}^{\Pi}(\bar A_K)
		\cup
		\{A_K^\Pi>\bar A_K\},
	\]
	for every \(K\le T\);
	\item the terminal relaxed quadratic-noise event \(\cE_{Q_T}\) and the
	terminal relaxed smoothing event \(\cE_{\alpha_T}\).
\end{enumerate}
Consequently, Eq.~\eqref{eq:c_del} below holds for every \(K\le T\):
\begin{equation}\label{eq:c_del}
\max\left\{
\log \frac{1}{\delta_K^{(Y)}},\;
\log \frac{1}{\delta_K^{(Q)}},\;
\log \frac{1}{\delta_K^{(\alpha)}}
\right\}
\leq c_\delta \Lambda.
\end{equation}
\end{lemma}

Lemma~\ref{lem:time_uniform_controls} places the weighted-scan event, the
linear-martingale implication, and the terminal nonnegative controls on a
single event of probability at least \(1-\delta\). The main induction is
carried out on this event.

\begin{remark}[Nonstandard elements of the time-uniform proof]
The two ingredients in Lemma~\ref{lem:time_uniform_controls} that depart from
standard analyses are the one-sided weighted lower scan in
Lemma~\ref{lem:cE_rho} and the product-martingale stitching in
Lemma~\ref{lem:product_linear_uniform}, which invokes
Lemma~\ref{lem:stitched_product_mg} to handle intrinsic variance. The
quadratic-noise and smoothing estimates are applied only at the terminal index
\(T\), while the product-linear variance proxy is verified from the accompanying
prefix angle control. Consequently, the proof avoids a direct
union bound over terminal times \(K=1,\dots,T\); for fixed problem parameters,
the confidence cost remains \(\log(1/\delta)+\log\log(e+T+T_0)\).
\end{remark}

\medskip

We now close the proof of Theorem~\ref{thm:main_1} by induction on the terminal
index.

\begin{proof}[\textbf{Proof of Theorem~\ref{thm:main_1}}]
	Let \(\mathcal G_T\) denote the event supplied by
	Lemma~\ref{lem:time_uniform_controls}, so that
	\(\Pr(\mathcal G_T)\ge1-\delta\). On \(\mathcal G_T\), all weighted-suffix
	and prefix angle controls hold, the product-linear implication
	\(\mathcal L_K(\bar A_K)\) holds for every \(K\le T\), and the terminal
	nonnegative controls \(\cE_{Q_T}\) and \(\cE_{\alpha_T}\) hold. In addition,
	all auxiliary logarithmic factors are bounded by
	\(c_\delta\Lambda=\Gamma_T(\delta)\), and the deterministic estimates in
	Lemma~\ref{lem:c_rho} are available for every \(K\le T\).

	We work deterministically on \(\mathcal G_T\) and proceed by induction on
	\(K\).
	\proofstep{Base case.}
	For \(K=0\), the initial-condition term in the definition of \(\cC\) in
	Eq.~\eqref{eq:cC}, together with \(c_\rho\ge1\), gives
	\[
	\Delta_0 \le \frac{\cC d \Theta_T}{T_0}.
	\]

	\proofstep{Induction hypothesis.}
	Fix \(K\in\{1,\dots,T\}\), and assume that, for every
	\(k=0,\dots,K-1\),
	\begin{equation}
		\Delta_k \le \frac{\cC d \Theta_T}{T_k}.
	\end{equation}

	On \(\mathcal G_T\), the unrolled recursion
	Eq.~\eqref{eq:dec_2} holds. We bound its four contributions in turn.

	\proofstep{Initial term.}
	Equation~\eqref{eq:c_rho} gives
	\(\rho_K\Delta_0\le c_\rho(T_0/T_K)\Delta_0\). Because
	\(\cC \ge 8c_\rho T_0 \Delta_0/(d\Theta_T)\) by Eq.~\eqref{eq:cC},
	\begin{equation}
		\label{eq:init_absorb}
		\rho_K \Delta_0
		\leq
		c_\rho \frac{T_0}{T_K}\Delta_0
		\le \frac{\cC}{8}\frac{d \Theta_T}{T_K}.
	\end{equation}

	\proofstep{Terminal comparison for nonnegative suffix terms.}
	For \(k<K\le T\), the definitions of \(\rho_{K,k}\) and \(\rho_{T,k}\)
	give
	\[
		\rho_{K,k}
		=
		\rho_{T,k}
		\exp\left(\sum_{t=K}^{T-1}\frac1{T_t}\right).
	\]
	Since
	\[
		\sum_{t=K}^{T-1}\frac1{T_t}
		\le
		1+\log\frac{T_T}{T_K},
	\]
	there is a universal constant \(C\) such that
	\begin{equation}\label{eq:terminal_suffix_compare}
		\rho_{K,k}\le C\frac{T_T}{T_K}\rho_{T,k},
		\qquad 0\le k<K\le T,
	\end{equation}
	and hence
	\begin{equation}\label{eq:terminal_suffix_compare_sq}
		\rho_{K,k}^2\le C
		\left(\frac{T_T}{T_K}\right)^2\rho_{T,k}^2,
		\qquad 0\le k<K\le T.
	\end{equation}
	Thus any nonnegative suffix-weighted sum at terminal time \(K\) is bounded
	by the corresponding terminal-\(T\) sum, with the appropriate power of
	\(T_T/T_K\).

	\proofstep{Signed linear product term.}

	The linear term in Lemma~\ref{lem:dec_2} is multiplied by the actual suffix
	product \(\Pi_{K,k}\), not by the deterministic upper bound
	\(\rho_{K,k}\). Since this term is signed, the replacement
	\(\Pi_{K,k}\le\rho_{K,k}\) is not order-preserving at the level of the
	linear sum. Instead, the event in
	Lemma~\ref{lem:time_uniform_controls} includes the bound from
	Lemma~\ref{lem:product_linear_uniform}, which controls the signed
	product-linear sum with the actual product \(\Pi_{K,k}\). It remains to
	verify the variance-proxy condition \(A_K^\Pi\le\bar A_K\).

	Set
	\[
		\bar\Delta_k=\frac{\cC d\Theta_T}{T_k},
		\qquad k=0,\dots,T-1,
	\]
	so that \(\Delta_k\le\bar\Delta_k\) throughout the induction range
	\(k<K\). Convexity and \(L\)-smoothness give the standard gradient bound
	\[
		f\!\left(\bx-\frac1L\nabla f(\bx)\right)
		\le
		f(\bx)-\frac{1}{2L}\|\nabla f(\bx)\|^2,
		\qquad
		\|\nabla f(\bx)\|^2\le2L\bigl(f(\bx)-f(\bx^*)\bigr),
	\]
	where the second inequality uses optimality of \(\bx^*\). Using this bound,
	\(\Pi_{K,k}\le\rho_{K,k}\), and the induction hypothesis,
	\[
	A_K^\Pi
	\le
		2L
		\sum_{k=0}^{K-1}
		\frac{\rho_{K,k}^2}{T_k^2}\bar\Delta_k\zeta_k.
	\]
	The squared terminal comparison in
	Eq.~\eqref{eq:terminal_suffix_compare_sq}, followed by
	\(\rho_{T,k}\le c_\rho T_k/T_T\), gives
	\[
		A_K^\Pi
		\le
		C
		\frac{L\cC c_\rho^2 d\Theta_T}{T_K^2}
		\sum_{k=0}^{T-1}\frac{\zeta_k}{T_k}.
	\]
	By the prefix estimate in Lemma~\ref{lem:time_uniform_controls},
	\[
		d\sum_{k=0}^{T-1}\frac{\zeta_k}{T_k}
		\le
		C\Lambda\left(1+\frac{c_\delta}{T_0}\right).
	\]
	Hence
	\[
		A_K^\Pi
		\le
		C
		\frac{L\cC c_\rho^2\Theta_T\Lambda}{T_K^2}
		\left(1+\frac{c_\delta}{T_0}\right)
		\le
		\bar A_K,
	\]
	after choosing the universal constant \(C_A\) in
	Lemma~\ref{lem:time_uniform_controls} large enough. Hence
	\(\mathcal L_K(\bar A_K)\) reduces to
	\(\cE_{Y_K}^{\Pi}(\bar A_K)\).
	Therefore, on \(\mathcal G_T\),
	\begin{equation}
	\label{eq:Y_up}
	\frac{4d}{\mu}
	\sum_{k=0}^{K-1}
	\Pi_{K,k}
	\frac{
	\nabla f(\bx_k)^\top\bu_k\,\bu_k^\top\be_k
	}{
	T_k\|\bu_k\|^2
	}
	\le
	C\frac{d\sigma c_\rho\Theta_T}{\mu T_K}
	\sqrt{
	L\cC
	\left(1+\frac{c_\delta}{T_0}\right)}
	\stackrel{\eqref{eq:cC}}{\le}
	\frac{\cC d\Theta_T}{8T_K}.
	\end{equation}

	\proofstep{Quadratic-noise term.}

	On \(\mathcal G_T\), the event \(\cE_{Q_T}\) holds. Thus, using
	Eq.~\eqref{eq:c_del}, Lemma~\ref{lem:c_rho}, and Lemma~\ref{lem:T_sum},
	\begin{align*}
	\sum_{k=0}^{T-1} \frac{ \rho_{T,k} }{ T_k^2} \cdot \frac{|\bu_k^\top \be_k|^2}{\norm{\bu_k}^2}
	&\le
	\frac{2\sigma^2}{d}
	\sum_{k=0}^{T-1}\frac{\rho_{T,k}}{T_k^2}
	+
	2\sigma^2
	\log\frac1{\delta_T^{(Q)}}
	\max_{0\le k\le T-1}\frac{\rho_{T,k}}{T_k^2}
	\\
	&\le
	\frac{2\sigma^2 c_\rho\Lambda}{dT_T}
	+
	\frac{2\sigma^2 c_\rho c_\delta\Lambda}{T_TT_0}.
	\end{align*}
	Therefore Eq.~\eqref{eq:terminal_suffix_compare} implies, for the current
	induction terminal time \(K\),
	\begin{equation}
		\label{eq:Q_up}
	\begin{aligned}
		&\frac{32Ld^2}{\mu^2}
		\sum_{k=0}^{K-1}
		\frac{\rho_{K,k}}{T_k^2}
		\cdot
		\frac{|\bu_k^\top \be_k|^2}{\|\bu_k\|^2}
		\leq
		C\frac{Ld\sigma^2c_\rho\Lambda}{\mu^2T_K}
		\left(
		1
		+
		\frac{d c_\delta}{T_0}
		\right)
		\stackrel{\eqref{eq:cC}}{\leq}
		\frac{\cC \cdot d\cdot \Theta_T}{8T_K}.
	\end{aligned}
	\end{equation}

	\proofstep{Smoothing-bias term.}
	Let
	\[
		c_1:=\frac{dL^2}{2\mu},
		\qquad
		c_2:=\frac{4d^2L^3}{\mu^2},
		\qquad
		a_k:=
		\frac{c_1}{T_k}
		+
		\frac{c_2}{T_k^2}.
	\]
	We invoke Lemma~\ref{lem:alp_term_up} only at the terminal index \(T\). On
	\(\mathcal G_T\), write
	\[
		A_T:=\sum_{k=0}^{T-1}\rho_{T,k}a_k,\qquad
		B_T:=\sum_{k=0}^{T-1}\rho_{T,k}^2a_k^2,\qquad
		M_T:=\max_{0\le k\le T-1}\rho_{T,k}a_k.
	\]
	Then Lemma~\ref{lem:alp_term_up} and Eq.~\eqref{eq:c_del} give
	\[
		\sum_{k=0}^{T-1}\rho_{T,k}a_k\|\bu_k\|^2
		\le
		Cd\left(A_T+\sqrt{\Gamma_T(\delta)B_T}
		+\Gamma_T(\delta)M_T\right).
	\]
	Using \(\rho_{T,k}\le c_\rho T_k/T_T\),
	Lemma~\ref{lem:T_sum}, \(T_T\ge T_0\), and
	\(\Gamma_T(\delta)=c_\delta\Lambda\), the three terms satisfy
	\[
	\begin{aligned}
		dA_T
		&\le
		Cc_\rho d
		\left(c_1+\frac{c_2\Lambda}{T_0}\right),
		\\
		d\sqrt{\Gamma_T(\delta)B_T}
		&\le
		Cc_\rho d\sqrt{c_\delta\Lambda}
		\left(\frac{c_1}{\sqrt{T_0}}+\frac{c_2}{T_0^{3/2}}\right),
		\\
		d\Gamma_T(\delta)M_T
		&\le
		Cc_\rho d c_\delta\Lambda
		\frac{1}{T_0}
		\left(c_1+\frac{c_2}{T_0}\right).
	\end{aligned}
	\]
	Substituting \(c_1=dL^2/(2\mu)\), \(c_2=4d^2L^3/\mu^2\), and
	\(T_0=32dL/\mu\), and using \(d,\Lambda\ge1\) and
	\(\sqrt{c_\delta\Lambda}\le \Lambda(1+c_\delta)\), yields
	\[
		\sum_{k=0}^{T-1}\rho_{T,k}a_k\|\bu_k\|^2
		\le
		Cc_\rho d^2\Lambda(1+c_\delta)
		\left(1+\frac{L^2}{\mu}+\frac{L^3}{\mu^2}\right),
	\]
	after increasing the universal numerical constant if necessary. Hence, by
	Eq.~\eqref{eq:terminal_suffix_compare},
	\begin{equation}
	\label{eq:alp_up}
	\alpha^2
	\sum_{k=0}^{K-1}\rho_{K,k}
	\left(
	\frac{dL^2}{2\mu T_k}
	+
	\frac{4d^2L^3}{\mu^2 T_k^2}
	\right)\|\bu_k\|^2
	\leq
	C\alpha^2\frac{T_T}{T_K}
	c_\rho d^2\Lambda(1+c_\delta)
	\left(1+\frac{L^2}{\mu}+\frac{L^3}{\mu^2}\right)
	\leq
	\frac{\cC d \Theta_T}{8 T_K},
	\end{equation}
	where the last inequality uses \(\alpha^2=1/(dT_T)\),
	\(\Theta_T=\Lambda\Gamma_T(\delta)\), and the last term in the definition
	of \(\cC\) in Eq.~\eqref{eq:cC}, after increasing the universal numerical
	constant if necessary.

	\proofstep{Conclusion of the induction step.}

	Combining Eq.~\eqref{eq:dec_2}, \eqref{eq:init_absorb},
	\eqref{eq:Y_up}, \eqref{eq:Q_up},
	and \eqref{eq:alp_up}, we conclude that
	\[
	\Delta_K
	\le
	\left(
	\frac{\cC}{8}
	+
	\frac{\cC}{8}
	+
	\frac{\cC}{8}
	+
	\frac{\cC}{8}
	\right)
	\frac{ d\Theta_T}{T_K}
	\le
	\cC \frac{ d\Theta_T}{T_K}.
	\]
	This proves the induction step.
	Since the induction is deterministic on \(\mathcal G_T\), it holds
	simultaneously for all \(K\le T\). Taking \(K=T\) yields the first inequality
	in Eq.~\eqref{eq:improved_rate_weighted}. To obtain the displayed
	big-\(\mathcal O\) form, multiply each term in the maximum defining \(\cC\)
	by \(d\Theta_T/(T+T_0)\), use
	\(\Theta_T=\Lambda\Gamma_T(\delta)\),
	\(T_0=32dL/\mu\), and \(c_\rho\ge1\), and absorb fixed problem parameters
	and universal numerical factors into the implicit constant.
\end{proof}

\medskip
\begin{proof}[Proof of Corollary~\ref{cor:main_1}]
	By Eq.~\eqref{eq:improved_rate_weighted}, with probability at least
	\(1-\delta\),
	\[
		f(\bx_T)-f(\bx^*)
		\le
		C
		\left(
		\frac{c_\rho T_0\Delta_0}{T+T_0}
		+
		c_\rho^2
		\frac{d\Lambda\Gamma_T(\delta)}{T+T_0}
		\left(1+\frac{\Gamma_T(\delta)}{T_0}\right)
		\right).
	\]
	By the definition of \(c_\rho\), the assumption
	\(T_0\gtrsim\Gamma_T(\delta)\) gives \(c_\rho=\mathcal O(1)\). Absorbing
	this numerical bound into the constant yields
	Eq.~\eqref{eq:simplified_rate}. The final
	\(\widetilde{\mathcal O}(d/T)\) display is obtained from the same estimate
	after fixed problem parameters and logarithmic factors are suppressed.
\end{proof}

\section{A General Pathwise Framework}
\label{sec:pathwise_template}

The proof above should not be read as a concentration inequality tailored only
to one Gaussian finite-difference theorem, or even only to zeroth-order
algorithms. The main contribution is a pathwise framework for stochastic
recursions with random contraction and signed perturbations. Whenever an error
process admits the recursion below and the associated scan, martingale, variance,
and terminal controls can be verified, the same proof modules apply. Zeroth-order
direction laws and oracle perturbation models are important instances of this
general structure.

\subsection{The Framework Recursion}

The basic object is the following one-step inequality:
\begin{equation}
	\label{eq:generic_pathwise_recursion}
	\Delta_{t+1}
	\le
	(1-\gamma_tZ_t)\Delta_t+\ell_t+Q_t+B_t.
\end{equation}
Here \(Z_t\) is the realized random contraction strength, \(\ell_t\) is a
signed perturbation term, and \(Q_t\) and \(B_t\) are nonnegative residual terms.
In zeroth-order applications, these objects correspond respectively to the
directional alignment, the signed oracle-noise contribution, the quadratic-noise
term, and the finite-difference or smoothing bias. The purpose of the analysis
is not merely to bound each one-step term, but to control the entire path
generated by
Eq.~\eqref{eq:generic_pathwise_recursion}. For the algebraic unrolling of the
upper bound, the realized factors
\[
	q_t:=1-\gamma_tZ_t
\]
must be nonnegative. For the product-martingale transformation based on
\(P_K/P_{k+1}\), one normally assumes the stronger stability condition
\[
	0<\underline q\le q_t\le1
\]
almost surely, so that the prefix products and their reciprocals are globally
defined. If stability is available only on a high-probability event, the same
argument should be applied to the stopped process with
\(\tau=\inf\{t:q_t\notin[\underline q,1]\}\); on the stability event the
stopping time is never reached. In the main theorem the condition is
deterministic, following from \(T_0=32dL/\mu\), \(L\ge\mu\), and
\(0\le \zeta_t\le1\):
\[
	\gamma_tZ_t
	=
	\frac{2d}{T_t}\zeta_t
	\le
	\frac{2d}{T_0}
	=
	\frac{\mu}{16L}
	\le
	\frac1{16},
	\qquad
	q_t\in[15/16,1].
\]

In the Gaussian same-sample model studied in the main theorem,
\[
	\gamma_t=\frac{2d}{T_t},
	\qquad
	Z_t=\zeta_t
	=
	\frac{(\bu_t^\top\nabla f(\bx_t))^2}
	{\|\bu_t\|^2\|\nabla f(\bx_t)\|^2},
\]
and
\[
	\ell_t
	=
	\frac{4d}{\mu T_t}
	\frac{\nabla f(\bx_t)^\top\bu_t\,\bu_t^\top\be_t}{\|\bu_t\|^2}.
\]
The terms \(Q_t\) and \(B_t\) are respectively the quadratic
stochastic-gradient error and the smoothing-bias contribution displayed in
Eq.~\eqref{eq:dec_2}. Other models enter the framework by changing the law of
\(Z_t\), the scale and structure of \(\ell_t\), or the nonnegative terms
\(Q_t\) and \(B_t\).

\subsection{Unrolling and Proof Modules}

Unrolling Eq.~\eqref{eq:generic_pathwise_recursion} gives, for every terminal
time \(K\),
\begin{equation}
	\label{eq:generic_unrolling}
	\Delta_K
	\le
	\Pi^{\rm gen}_{K,-1}\Delta_0
	+
	\sum_{k=0}^{K-1}\Pi^{\rm gen}_{K,k}\ell_k
	+
	\sum_{k=0}^{K-1}\Pi^{\rm gen}_{K,k}Q_k
	+
	\sum_{k=0}^{K-1}\Pi^{\rm gen}_{K,k}B_k,
\end{equation}
where
\[
	\Pi^{\rm gen}_{K,k}
	:=
	\prod_{t=k+1}^{K-1}(1-\gamma_tZ_t),
	\qquad
	\Pi^{\rm gen}_{K,-1}
	:=
	\prod_{t=0}^{K-1}(1-\gamma_tZ_t).
\]
Empty products are equal to one.
Thus the proof decomposes into four reusable modules.

\begin{enumerate}[(i)]
	\item \emph{Random contraction module.} Prove the weighted scan estimates
	for the variables \(Z_t\). A lower suffix scan converts many weak random
	directional contractions into uniform control of the suffix products
	\(\Pi^{\rm gen}_{K,k}\). An upper prefix or energy scan is also needed to
	control inverse-product envelopes and the intrinsic variance in the signed
	product-linear module.

	\item \emph{Signed product-linear module.} Control
	\(\sum_k\Pi^{\rm gen}_{K,k}\ell_k\) while retaining the realized product
	weights until martingale cancellation has been used. Replacing
	\(\Pi^{\rm gen}_{K,k}\) by deterministic envelopes, or taking absolute
	values too early, destroys the cancellation that is needed for the sharp
	last-iterate scale. For Gaussian or spherical directions, one may use the
	angle-enlarged filtration justified by
	Lemma~\ref{lem:future_angle_innovations}. For direction laws whose
	conditional angle distribution depends on the current gradient orientation,
	such as Rademacher or coordinate directions, the natural-filtration product
	transformation described below is the safer formulation. In every case, the
	required conditional centering is post-direction and pre-noise: conditioning
	only before the current direction is drawn is generally not enough, because
	the product denominator depends on that direction.

	\item \emph{Nonnegative terminal module.} Control the \(Q_t\) and \(B_t\)
	sums by terminal nonnegative concentration estimates. Here deterministic
	product envelopes can be used safely because the terms are nonnegative.

	\item \emph{Induction module.} Combine the preceding bounds on a common
	high-probability event and close the last-iterate induction in \(K\).
\end{enumerate}

The natural-filtration form of the signed module is worth making explicit.
Set
\[
	P_0:=1,
	\qquad
	P_K:=\prod_{t=0}^{K-1}q_t,
	\qquad
	q_t:=1-\gamma_tZ_t.
\]
Then
\[
	\Pi^{\rm gen}_{K,k}=\frac{P_K}{P_{k+1}},
	\qquad
	\sum_{k=0}^{K-1}\Pi^{\rm gen}_{K,k}\ell_k
	=
	P_K\sum_{k=0}^{K-1}\frac{\ell_k}{P_{k+1}}.
\]
Let \(\mathcal G_k\) denote the sigma-field generated by the information
available after the current direction is sampled and before the current oracle
noise is sampled. The
natural-filtration transformation requires
\[
	q_k,\ P_{k+1}\ \text{are }\mathcal G_k\text{-measurable},
	\qquad
	\EE[\ell_k\mid\mathcal G_k]=0.
\]
Then
\[
	\EE\left[
	\frac{\ell_k}{P_{k+1}}
	\middle|\mathcal G_k
	\right]=0.
\]
Equivalently, with the interlaced filtration that records \(\mathcal G_k\)
before the \(k\)-th noise draw and the post-noise history after it, the
variables \(\ell_k/P_{k+1}\) are martingale differences. If, for some
\(\mathcal G_k\)-measurable variance proxy \(\mathfrak v_k\),
\[
	\EE\!\left[
	\exp(\lambda\ell_k)
	\middle|\mathcal G_k
	\right]
	\le
	\exp(C\lambda^2\mathfrak v_k),
\]
then \(P_{k+1}\) is fixed under \(\mathcal G_k\), and the transformed
increment has variance proxy \(\mathfrak v_k/P_{k+1}^2\). This is the
product-martingale structure used by the signed module.
This formulation avoids revealing future contractions and is the appropriate
route for non-rotational direction laws.
The Gaussian proof in this paper uses an angle-enlarged version because the
future beta angles have a conditional law independent of the current oracle
noise, but that independence is a special feature of rotationally invariant
directions.

In the present paper, Lemma~\ref{lem:weighted_scan_bounds} supplies the
random contraction module, Lemma~\ref{lem:product_linear_uniform} supplies the
signed product-linear module, the quadratic-noise and smoothing estimates
supply the terminal module, and the proof of Theorem~\ref{thm:main_1} closes
the induction. The same decomposition is the reason the argument can be
transferred to other random-contraction settings.

\subsection{Changing the Direction Law}

The first natural extension changes the direction distribution. The pathwise
architecture is reusable, but the scan, normalized-projection concentration,
and direction-norm bookkeeping must be verified for the new direction law.
For the contraction scan, define the predictable unit vector
\[
	\bv_t
	:=
	\begin{cases}
	\nabla f(\bx_t)/\|\nabla f(\bx_t)\|, & \nabla f(\bx_t)\ne0,\\
	\be_1, & \nabla f(\bx_t)=0.
	\end{cases}
\]
Here and below, \(\be_i\) with an integer subscript denotes the \(i\)-th
standard basis vector, not the oracle-noise vector \(\be_t\).
This convention is harmless: if \(\nabla f(\bx_t)=0\), then strong convexity
implies \(\Delta_t=0\), so the contraction term multiplying \(\Delta_t\) is
irrelevant at that step. The convention is useful for discrete direction laws,
where the continuous nonvanishing-gradient argument used for Gaussian
directions need not apply. For a general randomized direction mechanism, assume
\(\|\bu_t\|>0\) almost surely, or define \(Z_t\) on the exceptional
zero-direction event by a harmless convention. Then write
\[
	Z_t
	:=
	\frac{\langle \bv_t,\bu_t\rangle^2}{\|\bu_t\|^2}.
\]
The contraction part of the framework asks for two complementary scan
estimates, stated for constants \(c,C>0\) and nonnegative weights
\(w_t,\omega_t\) that are \(\mathcal F_t^-\)-measurable. The first is a
time-uniform weighted lower suffix scan. The nonnegative quantities
\(\mathrm{Fluc}_{j,K}(\delta)\) and \(\mathrm{Fluc}^{+}_{K}(\delta)\) below are
fluctuation budgets supplied by the direction-specific scan lemma; their
dependence on the chosen weights, horizon, and direction law is suppressed in
the notation:
\begin{equation}
	\label{eq:generic_direction_scan}
	\sum_{t=j}^{K-1}w_tZ_t
	\ge
	\frac{c}{d}\sum_{t=j}^{K-1}w_t
	-
	\mathrm{Fluc}_{j,K}(\delta),
	\qquad 0\le j<K\le T.
\end{equation}
For suffix-product comparison, the relevant weights are the contraction
weights \(w_t=\gamma_t\), up to harmless deterministic constants; one takes
\(j=k+1\) to control \(\Pi^{\rm gen}_{K,k}\) and \(j=0\) for the initial
product. The second is an upper prefix or energy scan, for example
\begin{equation}
	\label{eq:generic_direction_upper_scan}
	\sum_{t=0}^{K-1}\omega_tZ_t
	\le
	\frac{C}{d}\sum_{t=0}^{K-1}\omega_t
	+
	\mathrm{Fluc}^{+}_{K}(\delta),
	\qquad 1\le K\le T,
\end{equation}
with weights \(\omega_t\) dictated by the inverse-product envelope or by the
intrinsic variance calculation. In the present proof these upper controls are
used, for example, to bound \(P_K^{-2}\) and to verify
\(A_K^\Pi\le\bar A_K\). Thus Eq.~\eqref{eq:generic_direction_scan} alone is
not enough for the framework.

This is not the only direction-dependent verification. The quadratic-noise
module also uses a normalized-projection estimate, which in angularly isotropic
cases is based on
\[
	\EE\left[
	\frac{\bu_t\bu_t^\top}{\|\bu_t\|^2}
	\middle|\mathcal F_t^-
	\right]
	=
	\frac{I_d}{d},
\]
or an appropriate analogue. The smoothing-bias module depends on the moments
or bounds of \(\|\bu_t\|^2\). Thus changing the sketch law changes the scan
lemma and may also change the projection and norm estimates, even though the
suffix-product architecture remains the same.

\paragraph{Spherical directions.}
If \(\bu_t\) is uniform on the sphere, then \(Z_t\) has the same
\(\mathrm{Beta}(1/2,(d-1)/2)\) law as the normalized Gaussian direction. Since
the norm-normalized stepsize already removes the Gaussian radius from the main
contraction factor, the spherical version is essentially a direct replacement
with simpler norm bookkeeping. The same rotational invariance also preserves
the future-angle independence used by the angle-enlarged product martingale.

\paragraph{Rademacher directions.}
For \(\bu_t\in\{\pm1\}^d\) with independent signs,
\[
	\EE[Z_t\mid\mathcal F_t^-]=\frac1d.
\]
The exact beta law and rotational invariance are no longer available, but
Khintchine-type bounds control the conditional moments of
\(\langle \bv_t,\bu_t\rangle\), and hence the weighted differences
\(Z_t-1/d\). More explicitly, since \(\|\bu_t\|^2=d\),
\[
	\EE[Z_t^2\mid\mathcal F_t^-]
	=
	\frac{3-2\sum_{i=1}^d v_{t,i}^4}{d^2}
	\le
	\frac3{d^2},
\]
and
\[
	\|Z_t-1/d\|_{\psi_1\mid\mathcal F_t^-}
	\lesssim
	\frac1d.
\]
This \(1/d\)-scale sub-exponential control is the relevant Bernstein scale for
the prefix and suffix scans; a bounded-increment argument using only
\(0\le Z_t\le1\) can lose the desired dimension dependence. This changes
constants and the scan proof, but not the suffix-product unrolling or the
signed product-martingale module. However, the conditional law of \(Z_t\) now
depends on the orientation \(\bv_t\). Therefore the future-angle independence
argument used for Gaussian directions should not be assumed; the
natural-filtration product transformation is the appropriate replacement.

\paragraph{Random coordinate directions.}
For \(\bu_t=\pm \be_{I_t}\), with \(I_t\) uniform on \(\{1,\dots,d\}\),
\[
	Z_t=v_{t,I_t}^2,
	\qquad
	\EE[Z_t\mid\mathcal F_t^-]=\frac1d.
\]
The mean contraction is still \(1/d\), but the fluctuations depend on the
coordinate coherence of \(\bv_t\):
\[
	\EE[Z_t^2\mid\mathcal F_t^-]
	=
	\frac1d\sum_{i=1}^d v_{t,i}^4,
	\qquad
	0\le Z_t\le \|\bv_t\|_\infty^2.
\]
The correct scan is therefore coherence-sensitive, with variance and
maximum-increment terms depending on quantities such as
\(\sum_i v_{t,i}^4\) and \(\|\bv_t\|_\infty^2\). When the descent direction is
spread across coordinates, the resulting scan resembles the Gaussian or
Rademacher case. When the direction is highly coherent, contraction arrives in
coordinate-dependent bursts, and the high-probability constants must reflect
that geometry.

This is the sketching interpretation of the framework. The current paper
proves the scan for Gaussian sketches because the beta law gives a clean,
dimension-explicit route. Other sketches should be analyzed by proving their
own versions of Eqs.~\eqref{eq:generic_direction_scan} and
\eqref{eq:generic_direction_upper_scan}; the later modules can then be reused,
preferably through the natural-filtration product transformation above.

\subsection{Additive Function-Value Noise as a Perturbation Module}

The same framework also accommodates changes in the oracle perturbation module:
one rederives the signed term, the nonnegative terminal terms, and the
corresponding intrinsic variance proxy, while keeping the random-contraction
unrolling intact. Additive function-value noise is a representative example.
Suppose, schematically, that
\[
	Y_t^\pm=f(\bx_t\pm\alpha_t\bu_t)+\upsilon_t^\pm,
\]
where, conditionally on the current history and on \(\bu_t\), the noises
\(\upsilon_t^+\) and \(\upsilon_t^-\) are mean-zero sub-Gaussian random
variables with common scale \(\tau_{\upsilon}\). For simplicity, assume that
the difference \(\upsilon_t^+-\upsilon_t^-\) is conditionally
sub-Gaussian at the corresponding scale; this is the case, for example, when
the two additive noises are conditionally independent. Assume also that
\(\alpha_t\) is \(\mathcal F_t^u\)-measurable, \(\alpha_t>0\) a.s., and that
the current additive noise is independent of future directions. Define
\[
	\nu_t
	:=
	\frac{\upsilon_t^+-\upsilon_t^-}{2\alpha_t}.
\]
The noise assumption can be stated directly as
\[
	\EE\left[
	\exp(\lambda\nu_t)
	\middle|\mathcal F_t^u
	\right]
	\le
	\exp\left(
	C\lambda^2\frac{\tau_{\upsilon}^2}{\alpha_t^2}
	\right),
	\qquad \lambda\in\mathbb R.
\]
Equivalently,
\[
	\|\nu_t\|_{\psi_2\mid\mathcal F_t^u}
	\lesssim
	\frac{\tau_{\upsilon}}{\alpha_t},
	\qquad
	\left\|
	\nu_t^2-\EE[\nu_t^2\mid\mathcal F_t^u]
	\right\|_{\psi_1\mid\mathcal F_t^u}
	\lesssim
	\frac{\tau_{\upsilon}^2}{\alpha_t^2}.
\]
Consequently,
\begin{equation}
	\label{eq:additive_noise_vector_second_moment}
	\EE\left[
	\left\|
	\frac{\upsilon_t^+-\upsilon_t^-}{2\alpha_t}\bu_t
	\right\|^2
	\middle|\,\mathcal F_t^u
	\right]
	\le
	C\frac{\tau_{\upsilon}^2\|\bu_t\|^2}{\alpha_t^2}.
\end{equation}
Moreover, conditional sub-Gaussianity of \(\nu_t\) implies the conditional
sub-exponential tail estimates for \(\nu_t^2\) needed by the nonnegative
quadratic-noise module; a second-moment estimate such as
Eq.~\eqref{eq:additive_noise_vector_second_moment} is not by itself sufficient
for the high-probability \(Q_t\) bound.
Let
\[
	s_t:=\langle\nabla f(\bx_t),\bu_t\rangle,
	\qquad
	r_t^2:=\|\bu_t\|^2.
\]
The finite-difference estimator can be written as
\[
	\widehat{\bg}_t
	=
	\bigl(
	s_t+b_t+\nu_t
	\bigr)\bu_t,
	\qquad
	|b_t|\le \frac{L\alpha_t}{2}r_t^2.
\]
Here \(b_t\) is the scalar finite-difference bias term.
By smoothness,
\[
\begin{aligned}
	\Delta_{t+1}
	\le\;&
	\Delta_t
	-\eta_t s_t(s_t+b_t+\nu_t)
	+
	\frac{L\eta_t^2r_t^2}{2}(s_t+b_t+\nu_t)^2.
\end{aligned}
\]
Using
\[
	(s+b+\nu)^2\le4s^2+2b^2+4\nu^2,
	\qquad
	-sb\le\frac{s^2+b^2}{2},
\]
and assuming \(\eta_t\le1/(8Lr_t^2)\), we get
\[
\begin{aligned}
	\Delta_{t+1}
	\le\;&
	\Delta_t
	-\frac{\eta_t}{4}s_t^2
	-\eta_t\nu_t s_t
	+
	2L\eta_t^2r_t^2\nu_t^2
	+
	\left(\frac{\eta_t}{2}+L\eta_t^2r_t^2\right)b_t^2.
\end{aligned}
\]
For the norm-normalized choice used below, this stepsize condition follows from
\[
	T_t=t+T_0,\qquad T_0\ge\frac{32dL}{\mu},
	\qquad
	L\eta_t\|\bu_t\|^2=\frac{4dL}{\mu T_t}\le\frac18.
\]
With the convention for \(\bv_t\) above, the identity
\(s_t^2=\|\nabla f(\bx_t)\|^2\|\bu_t\|^2Z_t\) holds whenever
\(\nabla f(\bx_t)\ne0\). If \(\nabla f(\bx_t)=0\), then strong convexity gives
\(\Delta_t=0\) and \(s_t=0\), so the same contraction inequality is trivial.
The Polyak--Lojasiewicz consequence of strong convexity,
\(\|\nabla f(\bx_t)\|^2\ge2\mu\Delta_t\), therefore gives
\[
	\Delta_{t+1}
	\le
	\left(
	1-\frac{\mu\eta_t}{2}\|\bu_t\|^2Z_t
	\right)\Delta_t
	-\eta_t\nu_t s_t
	+Q_t+B_t,
\]
where one may define
\[
	Q_t:=2L\eta_t^2\|\bu_t\|^2\nu_t^2.
\]
With the norm-normalized stepsize
\(\eta_t=4d/[\mu T_t\|\bu_t\|^2]\),
\[
	\Delta_{t+1}
	\le
	\left(1-\frac{2d}{T_t}Z_t\right)\Delta_t
	-\eta_t\nu_t s_t
	+Q_t+B_t,
	\qquad
	Q_t
	=
	\frac{32Ld^2}{\mu^2T_t^2}
	\frac{\nu_t^2}{\|\bu_t\|^2}.
\]
Moreover, the bias bound gives
\[
	B_t
	:=
	\left(\frac{\eta_t}{2}+L\eta_t^2\|\bu_t\|^2\right)b_t^2
	\le
	\frac{dL^2\alpha_t^2\|\bu_t\|^2}{2\mu T_t}
	+
	\frac{4d^2L^3\alpha_t^2\|\bu_t\|^2}{\mu^2T_t^2}.
\]
After unrolling, the additive-noise product-linear term becomes
\begin{equation}
	\label{eq:additive_product_linear_term}
	\sum_{k=0}^{K-1}
	\Pi^{\rm gen}_{K,k}
	\bigl(
	-\eta_k\nu_k s_k
	\bigr).
\end{equation}
This term has the same structural features as the signed suffix-product term
in the main proof: it is signed, it is coupled with the current random
direction, it is weighted by future random contractions, and it cannot be
safely bounded term-by-term in absolute value.
For non-rotational direction laws, independence of the future raw directions
from the current additive noise is not enough to justify revealing future
contractions, because their conditional laws may depend on future gradient
orientations. In that case the natural-filtration product transformation above
should be used.

The product-martingale module still applies, but with a different intrinsic
variance proxy. Up to universal sub-Gaussian constants, a valid proxy is
bounded by
\[
	A^{\Pi}_{K,\mathrm{add}}
	\lesssim
	\sum_{k=0}^{K-1}
	(\Pi^{\rm gen}_{K,k})^2
	\eta_k^2
	\frac{\tau_{\upsilon}^2}{\alpha_k^2}
	\bigl(\nabla f(\bx_k)^\top\bu_k\bigr)^2.
\]
For the norm-normalized stepsize used in this paper, this yields
\[
	A^{\Pi}_{K,\mathrm{add}}
	\lesssim
	\frac{d^2\tau_{\upsilon}^2}{\mu^2}
	\sum_{k=0}^{K-1}
	\frac{(\Pi^{\rm gen}_{K,k})^2}{T_k^2\alpha_k^2}
	\frac{\|\nabla f(\bx_k)\|^2Z_k}{\|\bu_k\|^2}.
\]
The angle scan alone does not control the whole variance proxy when the
smoothing schedule or inverse radii vary with the current direction. It must be
combined with a separate inverse-radius and smoothing-schedule control. What
changes is the scale: additive function-value noise introduces the expected
\(\alpha_k^{-2}\) amplification, and exact norm normalization introduces an
additional \(1/\|\bu_k\|^2\) factor in both this variance proxy and the
quadratic term \(Q_t\). For Gaussian directions this factor is a real
technical issue:
\[
	\EE\|\bu_t\|^{-2}=\frac1{d-2},
	\qquad d>2,
\]
and the expectation is infinite at \(d=2\). Moreover, the inverse radius has
polynomial rather than exponential upper tails, so a lower-tail event alone can
introduce polynomial dependence on \(T/\delta\). A formal additive-noise theorem
should therefore handle this factor by a separate lower-tail control for
\(\|\bu_t\|\), by a clipped or regularized normalization, or by a stepsize rule
whose denominator is bounded away from zero. The statistical rate and optimal
smoothing schedule may therefore change, but the pathwise proof architecture
does not.

\subsection{Summary of the Framework}

The general workflow is therefore: derive Eq.~\eqref{eq:generic_pathwise_recursion}
with stable nonnegative contraction factors, prove the lower and upper weighted
scan estimates appropriate to the direction law, verify the filtration
compatibility for the signed noise, control the signed product-linear term
with a product martingale using the correct intrinsic variance, bound the
nonnegative terminal terms, and close the last-iterate induction. In this
sense, the Gaussian same-sample theorem is one instance of a broader framework
for random contraction recursions rather than an isolated concentration
argument.

\section{Conclusion}

We established a direct high-probability last-iterate guarantee for the
standard same-sample two-point Gaussian method with a norm-normalized stepsize.
Under conditional sub-Gaussian stochastic-gradient noise, the theorem applies
for \(d\ge2\) and keeps the product-weight factor
\(c_\rho=\exp\{O(\Gamma_T(\delta)/T_0)\}\) explicit. In the clean-rate regime
\(T_0\gtrsim\Gamma_T(\delta)\), the last iterate satisfies
\[
	f(\bx_T)-f(\bx^*)
	=
	\widetilde{\mathcal O}\!\left(\frac{d}{T}\right)
\]
with probability at least \(1-\delta\), up to the explicit problem parameters
and logarithmic factors in Theorem~\ref{thm:main_1} and
Corollary~\ref{cor:main_1}. The confidence dependence is logarithmic, and the
argument avoids both bounded-noise truncation and expectation-to-confidence
conversion.

The proof also gives a reusable pathwise decomposition. Weighted lower and upper
scans convert random one-direction progress into aggregate contraction and
provide the prefix or energy bounds needed by variance envelopes. A
product-martingale argument controls the signed linear error without discarding
cancellation, while terminal nonnegative estimates handle the quadratic-noise
and smoothing-bias contributions. Section~\ref{sec:pathwise_template} records
this decomposition as a general pathwise framework for stochastic recursions with
random contraction and signed perturbations: it makes explicit the stability,
predictability, filtration, variance, and terminal-control conditions under
which the same proof strategy applies.

This framework separates the Gaussian and zeroth-order-specific ingredients from
the parts of the analysis that are structural. For a new model, the main tasks
are to derive a recursion of the same form, prove the appropriate lower and
upper weighted scans for the realized contraction variables, verify the
martingale structure of the signed perturbation, and control the nonnegative
terminal terms and intrinsic variance. In zeroth-order applications, changing
the direction law requires replacing the Gaussian beta scan by suitable
direction-specific estimates, while additive function-value noise illustrates
how a different oracle perturbation changes the signed linear term, the
nonnegative quadratic term, and the intrinsic variance. These observations
identify the main technical work needed for future high-probability guarantees
for any stochastic recursion that fits the random-contraction template.

\pagebreak
\bibliography{ref}
\bibliographystyle{apalike2}

\appendix

\section{Useful Lemmas}

This appendix collects the concentration tools used in the proof. We state the
standard inequalities in the form required above and include the necessary
short derivations.

\begin{lemma}[Laurent--Massart bound~\citep{Laurent2000Adaptive}]
\label{lem:chi_all}
	Let \(X_1,\dots,X_K\stackrel{\mathrm{i.i.d.}}{\sim}\mathcal N(0,1)\), and
	let \(w_1,\dots,w_K\ge0\) be fixed weights. With probability at least
	\(1-\delta\),
	\begin{equation*}
		\sum_{k=1}^K w_kX_k^2
		\le
		\sum_{k=1}^K w_k
		+2\sqrt{
		\left(\sum_{k=1}^K w_k^2\right)
		\log\left(\frac1\delta\right)}
		+2\max_{1\le k\le K}w_k\log\left(\frac1\delta\right).
	\end{equation*}
	In particular, if \(X\sim\chi^2(k)\) with \(k>0\), then, for every
	\(\tau>0\),
	\begin{align*}
		\Pr\left(X\le k-2\sqrt{k\tau}\right)\le e^{-\tau}.
	\end{align*}
\end{lemma}

\begin{lemma}[Ville's inequality]
\label{lem:ville}
Let \(\{Z_n\}_{n=0}^{N}\) be a nonnegative supermartingale. Then, for every
\(y>0\),
\[
\mathbb P\left(\max_{0\le n\le N}Z_n\ge y\right)
\le
\frac{\mathbb E[Z_0]}{y}.
\]
In particular, if \(Z_0=1\) almost surely, the right-hand side equals \(1/y\).
\end{lemma}

We also use the following standard decomposition of the two-point stochastic
zeroth-order estimator.
\begin{lemma}
	\label{lem:g_decom}
	Suppose that \(f(\cdot;\xi)\) is \(L\)-smooth. Then the stochastic
	zeroth-order estimator \(\bg(\bx;\xi)\) defined in Eq.~\eqref{eq:sg}
	satisfies
	\begin{equation}\label{eq:sg1}
		\bg(\bx;\xi)
		=
		\bu\bu^\top\nabla f(\bx;\xi)
		+ \beta \cdot \bu,
	\end{equation}
	where \(|\beta|\le(L\alpha/2)\norm{\bu}^2\).
\end{lemma}

\begin{proof}
	By the fundamental theorem of calculus,
	\[
	\frac{f(\bx+\alpha\bu;\xi)-f(\bx-\alpha\bu;\xi)}{2\alpha}
	=
	\frac12\int_{-1}^{1}
	\bu^\top\nabla f(\bx+s\alpha\bu;\xi)\,ds.
	\]
	Hence
	\[
	\frac{f(\bx+\alpha\bu;\xi)-f(\bx-\alpha\bu;\xi)}{2\alpha}
	=
	\bu^\top\nabla f(\bx;\xi)+\beta,
	\]
	where
	\[
	\beta
	:=
	\frac12\int_{-1}^{1}
	\bu^\top\bigl(\nabla f(\bx+s\alpha\bu;\xi)-\nabla f(\bx;\xi)\bigr)\,ds.
	\]
	By \(L\)-smoothness,
	\[
	|\beta|
	\le
	\frac12\int_{-1}^{1}
	L|s|\alpha\|\bu\|^2\,ds
	=
	\frac{L\alpha}{2}\|\bu\|^2.
	\]
	Multiplying the scalar central difference by \(\bu\) gives
	Eq.~\eqref{eq:sg1}.
\end{proof}

\begin{lemma}
\label{lem:T_sum}
	Let \(T_0>1\), let \(T_k=k+T_0\) for \(k\ge0\), and fix integers
	\(0\le k<K\). Then
	\begin{equation}
		\log\frac{T_K}{T_k}
		\leq
		\sum_{t=k}^{K-1}\frac{1}{t+T_0}
		\le
		1+\log\frac{T_K}{T_k}
		\le
		1+\log T_K	,
	\end{equation}
	and
	\begin{equation}
		\sum_{t=k}^{K-1}\frac{1}{(t+T_0)^2}
		\le
		\frac{2}{T_k}.
	\end{equation}
\end{lemma}

\begin{proof}
	The function \(x\mapsto 1/(x+T_0)\) is positive and decreasing. Therefore
\[
\sum_{t=k}^{K-1}\frac1{t+T_0}
\ge
\int_k^K\frac{dx}{x+T_0}
=
\log\frac{T_K}{T_k}.
\]
For the upper bound,
\[
\sum_{t=k}^{K-1}\frac1{t+T_0}
\le
\frac1{T_k}
+
\int_k^K\frac{dx}{x+T_0}
\le
1+\log\frac{T_K}{T_k}
\le
1+\log T_K.
\]
Similarly,
\[
\sum_{t=k}^{K-1}\frac1{(t+T_0)^2}
\le
\frac1{T_k^2}
+
\int_k^\infty\frac{dx}{(x+T_0)^2}
=
\frac1{T_k^2}+\frac1{T_k}
\le
\frac2{T_k},
\]
because \(T_k>1\).
\end{proof}

\section{Properties of the Beta Distribution}

\begin{lemma}\label{lem:beta_projection}
	Let \(\bx\sim\mathcal N(0,\bm I_d)\), and let
	\(\ba\in\mathbb R^d\) be a fixed unit vector. Define
	\begin{equation*}
		Y:=\frac{\langle\bx,\ba\rangle^2}{\|\bx\|_2^2}.
	\end{equation*}
	For \(d\ge2\),
	\(Y\sim\mathrm{Beta}(1/2,(d-1)/2)\); for \(d=1\), \(Y=1\)
	almost surely.
\end{lemma}

\begin{proof}
	By rotational invariance of the standard Gaussian distribution, we may assume
\(\ba=\be_1\). Then
\[
Y=\frac{X_1^2}{X_1^2+\cdots+X_d^2}.
\]
For \(d=1\), this ratio is one almost surely. For \(d\ge2\),
\(X_1^2\sim\chi^2_1\) and \(\sum_{i=2}^dX_i^2\sim\chi^2_{d-1}\) are
independent. Equivalently, these two variables are independent gamma random
variables with common scale parameter and shapes \(1/2\) and \((d-1)/2\).
The ratio of the first gamma variable to their sum therefore has the
\(\mathrm{Beta}(1/2,(d-1)/2)\) distribution.
\end{proof}

\begin{lemma}[Moments of the Beta distribution]
	\label{lem:beta}
	Let \(X\sim\mathrm{Beta}(\alpha,\beta)\), where
	\(\alpha,\beta>0\). Its mean and variance are
	\[
	\mathbb E[X]=\frac{\alpha}{\alpha+\beta},
	\qquad
	\mathrm{Var}(X)
	=
	\frac{\alpha\beta}{(\alpha+\beta)^2(\alpha+\beta+1)}.
	\]
	Moreover, for every integer \(m\ge1\), its \(m\)-th raw moment is
	\[
	\mathbb E[X^m]
	=
	\prod_{k=0}^{m-1}\frac{\alpha+k}{\alpha+\beta+k}.
	\]
\end{lemma}

\section{Product Martingales and Gaussian Angles}

\begin{lemma}[Stitched sub-Gaussian product boundary]
	\label{lem:stitched_product_mg}
	Let \((\mathcal H_k)_{k=0}^{T}\) be a filtration. Let
	\(q_k\in(0,1]\) and \(v_k\ge0\) be \(\mathcal H_k\)-measurable, let
	\(\vartheta_k\) be \(\mathcal H_{k+1}\)-measurable, \(P_0=1\), and
	\(P_K=\prod_{t=0}^{K-1}q_t\). Let
	\[
	\widetilde M_K
	=
	\sum_{k=0}^{K-1}\frac{\vartheta_k}{P_{k+1}},
	\qquad
	S_K=P_K\widetilde M_K,
	\qquad
	A_K=P_K^2\sum_{k=0}^{K-1}\frac{v_k}{P_{k+1}^2},
	\]
	where \(\vartheta_k\) is a martingale difference satisfying, conditionally on
	\(\mathcal H_k\),
	\[
	\EE[\vartheta_k\mid\mathcal H_k]=0,
	\qquad
	\EE[\exp(\lambda\vartheta_k)\mid\mathcal H_k]
	\le
	\exp(\lambda^2\sigma^2v_k),
	\qquad \lambda\in\mathbb R.
	\]
	Assume further that, for a deterministic number \(B_T\ge1\), on an event
	\(\mathcal P_T\),
	\[
	P_K^{-2}\le B_T,\qquad K=1,\dots,T,
	\]
	and let \(\bar A_K>0\) be deterministic variance envelopes. Set
	\[
	U_\star:=B_T\max_{K\le T}\bar A_K
	\]
	and assume that, for some \(0<\delta<1\) and confidence factor \(G\),
	\[
	G\ge 1+\log\frac1\delta,
	\qquad
	\max_{K\le T}
	\log\left(1+\log\left(e+\frac{U_\star}{\bar A_K}\right)\right)
	\le C\,G.
	\]
	Then, for a universal numerical constant \(C_{\mathrm{st}}\),
	\[
	\Pr\left(
	\mathcal P_T
	\cap
	\left\{
	\exists K\le T:
	A_K\le\bar A_K,\
	S_K>
	C_{\mathrm{st}}\sigma\sqrt{\bar A_K G}
	\right\}
	\right)
	\le\delta.
	\]
\end{lemma}

\begin{proof}
	\proofstep{Exponential supermartingale.}
	The variable \(P_{k+1}\) is \(\mathcal H_k\)-measurable, whereas
	\(\vartheta_k\) is \(\mathcal H_{k+1}\)-measurable. Hence the transformed
	increments \(\vartheta_k/P_{k+1}\) are adapted martingale differences. The
	conditional mgf assumption gives the same sub-Gaussian bound with variance
	proxy \(v_k/P_{k+1}^2\). Thus \(\widetilde M_K\) is a sub-Gaussian
	martingale with variance proxy
	\[
	\widetilde V_K:=\sum_{k=0}^{K-1}\frac{v_k}{P_{k+1}^2}.
	\]
	
	For every \(\lambda>0\),
	\[
	L_K(\lambda)
	:=
	\exp\left(
	\lambda\widetilde M_K
	-\lambda^2\sigma^2\widetilde V_K
	\right),
	\qquad K=0,\dots,T,
	\]
	is a nonnegative supermartingale. Indeed, after conditioning on
	\(\mathcal H_k\), the one-step multiplicative factor has expectation at most
	one by the preceding conditional mgf bound. Lemma~\ref{lem:ville} therefore
	gives, for every \(r>0\),
	\[
	\Pr\left(
	\exists K\le T:
	\lambda\widetilde M_K-\lambda^2\sigma^2\widetilde V_K\ge r
	\right)
	\le e^{-r}.
	\]
	
	\proofstep{Peeling over the intrinsic variance.}
	Fix a deterministic \(U>0\). For
	\(j=0,1,2,\dots\), set
	\[
	I_j:=\left(2^{-j-1}U,\,2^{-j}U\right],
	\qquad
	\delta_j:=\frac{6\delta}{\pi^2(j+1)^2},
	\qquad
	r_j:=\log\frac1{\delta_j},
	\]
	and choose
	\[
	\lambda_j:=
	\sqrt{\frac{r_j}{\sigma^2\,2^{-j}U}}.
	\]
	The intervals \(I_j\) partition \((0,U]\). On the
	complement of the Ville event for epoch \(j\), every \(K\le T\) with
	\(\widetilde V_K\in I_j\) satisfies
	\[
	\widetilde M_K
	\le
	\lambda_j\sigma^2\widetilde V_K+\frac{r_j}{\lambda_j}
	=
	\sigma\sqrt{r_j}
	\left(
	\frac{\widetilde V_K}{\sqrt{2^{-j}U}}
	+\sqrt{2^{-j}U}
	\right)
	\le
	3\sigma\sqrt{\widetilde V_K\,r_j}.
	\]
	The last inequality uses
	\(\widetilde V_K\le2^{-j}U\le2\widetilde V_K\) on \(I_j\). Moreover, on
	this epoch,
	\[
	r_j
	\le
	C\left\{
	\log\frac1\delta
	+
	\log\left(1+\log\left(e+\frac{U}{\widetilde V_K}\right)\right)
	\right\}.
	\]
	Since \(\sum_{j\ge0}\delta_j\le\delta\), a union bound over the
	variance epochs shows that, with probability at least \(1-\delta\),
	the following bound holds simultaneously for all \(K\le T\) with
	\(0<\widetilde V_K\le U\):
	\begin{equation}
		\label{eq:stitched_product_mg_raw_boundary}
		\widetilde M_K
		\le
		C\sigma\sqrt{\widetilde V_K
			\left(
			\log\frac1\delta+
			\log\left(1+\log\left(e+\frac{U}{\widetilde V_K}\right)\right)
			\right)}.
	\end{equation}
	The estimate is already uniform in the terminal time \(K\); Lemma~\ref{lem:ville}
	handled the maximum over \(K\) inside each epoch. If \(\widetilde V_K=0\),
	then all variance proxies up to time \(K\) are zero, and the conditional mgf
	bound with zero proxy forces the corresponding martingale increments to vanish
	almost surely; hence \(\widetilde M_K=0\).
	
	\proofstep{Transfer to deterministic variance envelopes.}
	Apply Eq.~\eqref{eq:stitched_product_mg_raw_boundary} with \(U=U_\star\).
	For every \(K\) that can contribute to the failure event---that is, every
	\(K\) satisfying \(A_K\le\bar A_K\)---the event \(\mathcal P_T\) implies
	\[
	\widetilde V_K=P_K^{-2}A_K
	\le
	B_T\bar A_K,
	\]
	and hence \(\widetilde V_K\le U_\star\). Multiplying
	Eq.~\eqref{eq:stitched_product_mg_raw_boundary} by \(P_K\) therefore gives
	\[
	S_K
	\le
	C\sigma
	\sqrt{
		A_K
		\left(
		\log\frac1\delta+
		\log\left(1+\log\left(e+\frac{U_\star}{\widetilde V_K}\right)\right)
		\right)}.
	\]
	It remains only to replace the random variance \(A_K\) by the deterministic
	envelope \(\bar A_K\). If \(A_K=0\), then \(\widetilde V_K=0\) and hence
	\(S_K=0\), so such a \(K\) cannot contribute to the failure event. We may
	therefore consider \(0<A_K\le\bar A_K\). Since \(q_t\le1\), \(P_K\le1\), so
	\(\widetilde V_K=P_K^{-2}A_K\ge A_K\); hence the logarithmic factor with
	\(U_\star/\widetilde V_K\) is no larger than the same factor with
	\(U_\star/A_K\). For \(0<a\le b\) and \(U>0\), we use the deterministic
	inequality
	\begin{equation}
	\label{eq:variance_envelope_loglog_transfer}
	a\log\left(1+\log\left(e+\frac{U}{a}\right)\right)
	\le
	Cb\left[
	1+\log\left(1+\log\left(e+\frac{U}{b}\right)\right)
	\right].
	\end{equation}
	To see this, write \(s=b/a\ge1\) and \(x=U/b\). Since
	\(e+sx\le s(e+x)\), we have
	\[
		\log\left(1+\log(e+sx)\right)
		\le
		C\left[
		\log\left(1+\log(e+x)\right)
		+
		\log\left(1+\log s\right)
		\right].
	\]
	Multiplying by \(a=b/s\) and using
	\(\sup_{s\ge1}s^{-1}\log(1+\log s)<\infty\) proves the displayed
	inequality, after increasing the universal constant \(C\). Applying
	Eq.~\eqref{eq:variance_envelope_loglog_transfer} with
	\(a=A_K\), \(b=\bar A_K\), and \(U=U_\star\), and also using
	\(A_K\log(1/\delta)\le\bar A_K\log(1/\delta)\), yields
	\[
	S_K
	\le
	C_{\mathrm{st}}\sigma\sqrt{\bar A_K\,G}
	\]
	for every \(K\le T\) with \(A_K\le\bar A_K\), where the assumed deterministic
	bound on
	\(\max_K\log(1+\log(e+U_\star/\bar A_K))\) and
	\(G\ge1+\log(1/\delta)\) were used in the final step. Thus the
	displayed failure event is contained in the complement of the stitched event
	constructed above, whose probability is at most \(\delta\).
\end{proof}


\begin{lemma}[Uniform weighted scan bounds]
\label{lem:weighted_scan_bounds}
	Let \(T\ge1\), \(d\ge1\), \(T_0\ge1\), and define
	\[
	w_t:=\frac{1}{t+T_0},
	\qquad
	\zeta_t:=\frac{(\bu_t^\top\ba_t)^2}{\|\bu_t\|^2},
	\qquad t=0,\dots,T-1,
	\]
	where \(\ba_t\) is an \(\cF_{t-1}\)-measurable unit vector.
	Assume that \((\cF_t)_{t\ge -1}\) is a filtration and that, for each
	\(t\), \(\bu_t\) is \(\cF_t\)-measurable while, conditionally on
	\(\cF_{t-1}\), \(\bu_t\sim\mathcal N(0,I_d)\) is independent of the
	past. Let
	\[
	J_T
	:=
	1+
	\left\lceil
	\log_2
	\left(
	\frac{T(T+T_0)}{T_0}
	\right)
	\right\rceil,
	\qquad
	\Gamma_T^{\rm scan}(\delta)
	:=
	1+\log\frac1\delta+\log J_T+\log(e+T_0).
	\]
	Then there is a universal numerical constant \(C>0\) such that, with
	probability at least \(1-\delta\), the following estimates hold
	simultaneously.
	First, for every interval \(0\le a\le b\le T-1\),
	\begin{equation}
	\label{eq:weighted_lower_scan_loglog}
		\sum_{t=a}^{b}w_t\zeta_t
		\ge
		\frac{1}{2d}\sum_{t=a}^{b}w_t
		-
		\frac{C}{dT_0}\Gamma_T^{\rm scan}(\delta).
	\end{equation}
	Second, for every \(1\le K\le T\),
	\begin{equation}
	\label{eq:weighted_prefix_upper_loglog}
		d\sum_{t=0}^{K-1}w_t\zeta_t
		\le
		C\left(1+\log(T+T_0)\right)
		+
		\frac{C}{T_0}\Gamma_T^{\rm scan}(\delta).
	\end{equation}
	Consequently, for every \(1\le K\le T\) and every
	\(k=-1,0,\dots,K-2\), with
	\[
	W_{K,k}:=\sum_{t=k+1}^{K-1}\frac1{t+T_0},
	\]
	one has
	\begin{equation}
	\label{eq:weighted_low_suffix_sum_corrected}
		\sum_{t=k+1}^{K-1}\frac{\zeta_t}{t+T_0}
		\ge
		\frac{W_{K,k}}{2d}
		-
		\frac{C}{dT_0}\Gamma_T^{\rm scan}(\delta).
	\end{equation}
\end{lemma}

\begin{proof}
	If \(d=1\), then \(\zeta_t=1\) almost surely. The lower bounds are
	immediate, while the prefix bound follows from
	\(\sum_{t=0}^{K-1}w_t\le 1+\log(T+T_0)\), after increasing \(C\). Hence
	we assume \(d\ge2\).

	\proofstep{Conditional moments of the angle variables.}
	Conditional on \(\cF_{t-1}\), the vector \(\ba_t\) is fixed and
	\(\bu_t\sim\mathcal N(0,I_d)\). Therefore
	\[
	\zeta_t\mid\cF_{t-1}
	\sim
	\mathrm{Beta}\left(\frac12,\frac{d-1}{2}\right),
	\]
	and Lemma~\ref{lem:beta} gives
	\[
	\EE[\zeta_t\mid\cF_{t-1}]=\frac1d,
	\qquad
	\mathrm{Var}(\zeta_t\mid\cF_{t-1})\le\frac{2}{d^2},
	\]
	together with the following centered moment estimate. For every integer
	\(p\ge1\),
	\[
	\EE[\zeta_t^p\mid\cF_{t-1}]
	=
	\prod_{i=0}^{p-1}\frac{i+1/2}{i+d/2}
	\le
	p!\left(\frac2d\right)^p.
	\]
	The last inequality uses \(i+d/2\ge d/2\) and
	\(\prod_{i=0}^{p-1}(2i+1)\le 2^p p!\).
	Thus, for every integer \(p\ge2\),
	\[
	\EE\left[
	\left|\zeta_t-\frac1d\right|^p
	\middle|\cF_{t-1}
	\right]
	\le
	p!\left(\frac4d\right)^p.
	\]
	Set
	\[
	Z_t:=w_t\left(\frac1d-\zeta_t\right).
	\]
	Then \(Z_t\) is conditionally centered,
	\[
	\EE[Z_t^2\mid\cF_{t-1}]
	\le
	\frac{2w_t^2}{d^2},
	\qquad
	Z_t\le\frac{w_t}{d}.
	\]

	\proofstep{A blockwise maximal inequality.}
	Fix a deterministic interval \([r,s]\) on which \(w_t\le m_0\), choose a sign
	\(\mathfrak s\in\{-1,1\}\), and write
	\(A_{r,u}:=\sum_{t=r}^{u}w_t\) and
	\[
	M_{r,u}^{(\mathfrak s)}
	:=
	\sum_{t=r}^{u}
	\mathfrak s\,w_t\left(\zeta_t-\frac1d\right).
	\]
	Let
	\(\mathcal B\subseteq\{r,\dots,s\}\) be a set of endpoints with largest
	element \(u_\star\), and suppose that
	\(A_{r,u_\star}\le2A_{r,u}\) for all \(u\in\mathcal B\). Then,
	for every \(\alpha\in(0,1)\), with probability at least \(1-\alpha\),
	\begin{equation}
	\label{eq:weighted_block_maximal}
		M_{r,u}^{(\mathfrak s)}
		\le
		\frac{A_{r,u}}{2d}
		+
		\frac{Cm_0}{d}\log\frac1\alpha,
		\qquad u\in\mathcal B.
	\end{equation}
	Indeed, for every integer \(p\ge2\), the centered moment estimate above
	implies
	\[
	\EE\!\left[
	\left|
	\mathfrak s\,w_t\left(\zeta_t-\frac1d\right)
	\right|^p
	\middle|\cF_{t-1}
	\right]
	\le
	\frac{p!}{2}
	\frac{32w_t^2}{d^2}
	\left(\frac{4m_0}{d}\right)^{p-2}.
	\]
	Hence, for \(0<\lambda<d/(4m_0)\),
	\[
	\EE\!\left[
	\exp\left\{
	\lambda\mathfrak s\,w_t\left(\zeta_t-\frac1d\right)
	\right\}
	\middle|\cF_{t-1}
	\right]
	\le
	\exp\left(
	\frac{\lambda^2(32w_t^2/d^2)}
	{2(1-4\lambda m_0/d)}
	\right).
	\]
	Therefore the standard exponential transform of
	\(M_{r,u}^{(\mathfrak s)}\), stopped at \(u_\star\), is a nonnegative
	supermartingale with variance proxy
	\[
	V_\star:=\frac{32}{d^2}\sum_{t=r}^{u_\star}w_t^2
	\le
	\frac{32m_0A_{r,u_\star}}{d^2}
	\]
	and Bernstein scale \(b_\star:=4m_0/d\). Applying
	Lemma~\ref{lem:ville} and optimizing over \(\lambda\) give
	\[
	\Pr\!\left(
	\max_{r\le u\le u_\star}M_{r,u}^{(\mathfrak s)}\ge x
	\right)
	\le
	\exp\left(-\frac{x^2}{2(V_\star+b_\star x)}\right).
	\]
	Thus, with probability at least \(1-\alpha\),
	\[
	\max_{r\le u\le u_\star}M_{r,u}^{(\mathfrak s)}
	\le
	\frac{C}{d}
	\left(
	\sqrt{m_0A_{r,u_\star}\log\frac1\alpha}
	+
	m_0\log\frac1\alpha
	\right).
	\]
	For \(u\in\mathcal B\), use \(A_{r,u_\star}\le2A_{r,u}\) and
	\(2\sqrt{xy}\le x/2+2y\) to absorb the square-root term into
	\(A_{r,u}/(2d)+Cm_0d^{-1}\log(1/\alpha)\), proving
	\eqref{eq:weighted_block_maximal}.

	\proofstep{Uniform lower scan.}
	Partition time into dyadic annuli
	\[
	\mathcal I_j
	:=
	\{t:2^jT_0\le t+T_0<2^{j+1}T_0\},
	\qquad j=0,1,\dots,J_T.
	\]
	Empty annuli are ignored. On \(\mathcal I_j\), set
	\(m_j:=(2^jT_0)^{-1}\); then \(w_t\le m_j\), while
	\(w_t\ge m_j/2\), and the number of possible starting points is at most
	\(N_j\le 2^jT_0+1\).

	Fix \(j\) and \(a\in\mathcal I_j\). For intervals
	\([a,b]\subseteq\mathcal I_j\), partition the endpoints by weighted mass:
	\[
	\mathcal B_{j,\ell}(a)
	:=
	\{b\in\mathcal I_j:b\ge a,\
	2^{\ell-1}m_j\le A_{a,b}<2^\ell m_j\},
	\qquad \ell\ge0.
	\]
	These blocks cover all admissible endpoints because \(A_{a,a}=w_a\ge
	m_j/2\). Moreover, \(\sum_{t\in\mathcal I_j}w_t\le C\); hence a nonempty
	block must satisfy \(2^{\ell-1}m_j\le C\), and therefore
	\[
	\ell\le C\bigl(\log(e+T_0)+j\bigr)
	\]
	after increasing the numerical constant \(C\). Thus only
	\(O(\log(e+T_0)+j)\) endpoint blocks are nonempty. If
	\(\mathcal B_{j,\ell}(a)\) is nonempty and \(b_{j,\ell}(a)\) is its largest
	endpoint, then
	\(A_{a,b_{j,\ell}(a)}<2A_{a,b}\) for every
	\(b\in\mathcal B_{j,\ell}(a)\). Apply
	\eqref{eq:weighted_block_maximal} with sign \(\mathfrak s=-1\), so that
	\(M_{a,b}^{(-1)}=\sum_{t=a}^{b}w_t(1/d-\zeta_t)\), and assign the failure
	probability
	\[
	\delta_{j,a,\ell}
	=
	\frac{c\delta}
	{(j+1)^2(N_j+1)(\ell+1)^2J_T},
	\]
	where \(c>0\) is a universal normalizing constant. Since
	\(|\mathcal I_j|\le N_j\), the total confidence budget is bounded by
	\[
	\sum_{j=0}^{J_T}\sum_{a\in\mathcal I_j}\sum_{\ell\ge0}
	\delta_{j,a,\ell}
	\le
	\frac{c\delta}{J_T}
	\sum_{j=0}^{J_T}\frac{1}{(j+1)^2}
	\frac{|\mathcal I_j|}{N_j+1}
	\sum_{\ell\ge0}\frac{1}{(\ell+1)^2}
	\le
	Cc\delta.
	\]
	Choosing \(c\) small enough makes this at most \(\delta/2\). On the resulting event,
	all intervals contained in a single annulus satisfy
	\[
	\sum_{t=a}^{b}w_t\zeta_t
	\ge
	\frac1{2d}\sum_{t=a}^{b}w_t
	-
	\frac{C}{d}
	\left(
	m_j\Gamma_T^{\rm scan}(\delta)
	+
	\frac{j+1}{2^jT_0}
	\right).
	\]
	Here we used
	\[
	m_j\log\frac1{\delta_{j,a,\ell}}
	\le
	Cm_j\Gamma_T^{\rm scan}(\delta)
	+
	\frac{C(j+1)}{2^jT_0},
	\]
	which follows from \(N_j\le2^jT_0+1\) and from the bound
	\(\ell\le C(\log(e+T_0)+j)\) for nonempty endpoint blocks.

	For an interval \([a,b]\) spanning multiple annuli, let
	\(j_0,\dots,j_1\) be the annuli it intersects. Decompose \([a,b]\) into the
	right tail of \(\mathcal I_{j_0}\), the full annuli
	\(\mathcal I_{j_0+1},\dots,\mathcal I_{j_1-1}\), and the left tail of
	\(\mathcal I_{j_1}\), omitting empty pieces. Applying the one-annulus bound
	to each piece, the leading weighted-mass terms sum to
	\((2d)^{-1}\sum_{t=a}^{b}w_t\). The accumulated error is at most
	\[
	\frac{C}{d}
	\sum_{j=j_0}^{j_1}
	\left(
	m_j\Gamma_T^{\rm scan}(\delta)
	+
	\frac{j+1}{2^jT_0}
	\right).
	\]
	Since each annulus is counted at most once and the deterministic errors are
	summable,
	\[
	\sum_{j\ge0}m_j\le\frac2{T_0},
	\qquad
	\sum_{j\ge0}\frac{j+1}{2^jT_0}\le\frac{C}{T_0}.
	\]
	Using \(\Gamma_T^{\rm scan}(\delta)\ge1\) and increasing \(C\), this proves
	\eqref{eq:weighted_lower_scan_loglog}.

	\proofstep{Uniform prefix upper bound.}
	Apply the same exponential-supermartingale
	estimate with sign \(\mathfrak s=1\) to the centered process
	\[
	R_K:=\sum_{t=0}^{K-1}w_t\left(\zeta_t-\frac1d\right).
	\]
	Since
	\[
	\sum_{t=0}^{T-1}w_t^2\le\frac2{T_0},
	\qquad
	w_0=\frac1{T_0},
	\]
	the preceding Bernstein--Ville calculation, now on the whole interval
	\([0,T-1]\), has variance proxy at most \(C/(d^2T_0)\) and Bernstein scale
	at most \(C/(dT_0)\). With probability at least \(1-\delta/2\), and using
	\(\Gamma_T^{\rm scan}(\delta)\ge\log(2/\delta)\),
	\[
	\max_{1\le K\le T}R_K
	\le
	\frac{C}{d}
	\left(
	\sqrt{\frac{\Gamma_T^{\rm scan}(\delta)}{T_0}}
	+
	\frac{\Gamma_T^{\rm scan}(\delta)}{T_0}
	\right).
	\]
	Using \(\sqrt{x}\le1+x\), adding
	\[
	\frac1d\sum_{t=0}^{K-1}w_t
	\le
	\frac{C(1+\log(T+T_0))}{d},
	\]
	and multiplying by \(d\), yields
	\eqref{eq:weighted_prefix_upper_loglog}. The suffix display
	\eqref{eq:weighted_low_suffix_sum_corrected} follows immediately from
	\eqref{eq:weighted_lower_scan_loglog} with \(a=k+1\) and \(b=K-1\).
\end{proof}

\section{Proofs of Section~\ref{sec:zsgd}}

\subsection{Proofs for the Initial Descent and Weighted-Product Recursion}

\begin{proof}[Proof of Lemma~\ref{lem:dec_1}]
	By the smoothness of \(f\) and the update rule of the algorithm,
	\begin{align*}
		&f(\bx_{t+1})
		\\
		\stackrel{\eqref{eq:L}}{\leq}&
		f(\bx_t) - \eta_t \dotprod{\nabla f(\bx_t), \bg(\bx_t;\xi_t)} + \frac{L\eta_t^2}{2} \norm{\bg(\bx_t;\xi_t)}^2 \\
		\stackrel{\eqref{eq:sg1}}{=}&
		f(\bx_t) - \eta_t \dotprod{\nabla f(\bx_t), \bu_t\bu_t^\top \nabla f(\bx_t;\xi_t) + \beta_t \bu_t} + \frac{L\eta_t^2}{2} \norm{\bu_t\bu_t^\top \nabla f(\bx_t;\xi_t) + \beta_t \bu_t}^2\\
		\leq&
		f(\bx_t) - \eta_t \dotprod{\nabla f(\bx_t), \bu_t\bu_t^\top \nabla f(\bx_t;\xi_t)} + \frac{\eta_t \beta_t^2}{2} + \frac{\eta_t ( \bu_t^\top \nabla f(\bx_t))^2}{2}\\
		&+
		L\eta_t^2 \left( \norm{\bu_t}^2 |\bu_t^\top \nabla f(\bx_t;\xi_t)|^2 + \beta_t^2 \norm{\bu_t}^2 \right)\\
		=&
		f(\bx_t) - \eta_t ( \bu_t^\top \nabla f(\bx_t))^2
		+ \eta_t \dotprod{\nabla f(\bx_t), \bu_t\bu_t^\top \be_t}
		\\
		&
		+ \frac{\eta_t \beta_t^2}{2}
		+ \frac{\eta_t ( \bu_t^\top \nabla f(\bx_t))^2}{2}
		+
		L\eta_t^2 \left( \norm{\bu_t}^2 |\bu_t^\top \nabla f(\bx_t;\xi_t)|^2 + \beta_t^2 \norm{\bu_t}^2 \right)\\
		\leq&
		f(\bx_t) - \eta_t ( \bu_t^\top \nabla f(\bx_t))^2
		+ \eta_t \dotprod{\nabla f(\bx_t), \bu_t\bu_t^\top \be_t}
		+ \frac{\eta_t \beta_t^2}{2}
		+ \frac{\eta_t ( \bu_t^\top \nabla f(\bx_t))^2}{2}
		\\
		&
		+
		L\eta_t^2 \left( 2\norm{\bu_t}^2 \Big( |\bu_t^\top \nabla f(\bx_t)|^2 + |\bu_t^\top \be_t|^2 \Big) + \beta_t^2 \norm{\bu_t}^2 \right)
		\\
		=&
		f(\bx_t) - \frac{\eta_t ( \bu_t^\top \nabla f(\bx_t))^2}{2}
		+ 2L\eta_t^2 \norm{\bu_t}^2 ( \bu_t^\top \nabla f(\bx_t))^2
		+ \frac{\eta_t \beta_t^2}{2}
		+ L\eta_t^2 \beta_t^2 \norm{\bu_t}^2\\
		&
		+ \eta_t \dotprod{\nabla f(\bx_t), \bu_t\bu_t^\top \be_t}
		+ 2L\eta_t^2\norm{\bu_t}^2|\bu_t^\top \be_t|^2.
	\end{align*}

	Consequently,
	\begin{align*}
		f(\bx_{t+1}) - f(\bx^*)
		\leq&
		f(\bx_t) - f(\bx^*)
		- \frac{\eta_t ( \bu_t^\top \nabla f(\bx_t))^2}{2}
		+ 2L\eta_t^2 \norm{\bu_t}^2 ( \bu_t^\top \nabla f(\bx_t))^2
		+ \Delta_{\alpha, t}\\
		&
		+ \eta_t \dotprod{\nabla f(\bx_t), \bu_t\bu_t^\top \be_t}
		+ 2L\eta_t^2\norm{\bu_t}^2|\bu_t^\top \be_t|^2\\
		=&
		f(\bx_t) - f(\bx^*)
		-
		\frac{\eta_t}{2} \left(1 - 4L\eta_t \norm{\bu_t}^2\right)( \bu_t^\top \nabla f(\bx_t))^2\\
		&
		+ \eta_t \dotprod{\nabla f(\bx_t), \bu_t\bu_t^\top \be_t}
		+
		2L\eta_t^2\norm{\bu_t}^2|\bu_t^\top \be_t|^2
		+ \Delta_{\alpha, t}\\
		\leq&
		f(\bx_t) - f(\bx^*)
		-
		\frac{\eta_t}{4}\bigl(\bu_t^\top\nabla f(\bx_t)\bigr)^2 \\
		&
		+ \eta_t \dotprod{\nabla f(\bx_t), \bu_t\bu_t^\top \be_t}
		+
		2L\eta_t^2\norm{\bu_t}^2|\bu_t^\top \be_t|^2
		+ \Delta_{\alpha, t}\\
		\leq&
		\left(1 - \frac{\mu\eta_t}{2}
		\frac{(\bu_t^\top\nabla f(\bx_t))^2}{\|\nabla f(\bx_t)\|^2}
		\right)\cdot \Big( f(\bx_t) - f(\bx^*) \Big)\\
		&
		+ \eta_t \dotprod{\nabla f(\bx_t), \bu_t\bu_t^\top \be_t}
		+
		2L\eta_t^2\norm{\bu_t}^2|\bu_t^\top \be_t|^2
		+ \Delta_{\alpha, t},
	\end{align*}
	where the second inequality uses
	\(\eta_t\le1/(8L\norm{\bu_t}^2)\), and the final inequality uses the
	Polyak--{\L}ojasiewicz inequality implied by \(\mu\)-strong convexity.

\end{proof}

\medskip

\begin{proof}[Proof of Lemma~\ref{lem:cE_rho}]
	Write
	\(
		w_t:=\frac1{T_t},
	\)
	and use the angle variable \(\zeta_t\) defined after
	Proposition~\ref{prop:nonvanishing_gradients}. To apply
	Lemma~\ref{lem:weighted_scan_bounds}, take the predictable field for the
	\(t\)-th summand to be the pre-direction history \(\mathcal F_t^{-}\), take
	the post-direction field to be \(\mathcal F_t^{u}\), and set
	\[
		\ba_t:=\frac{\nabla f(\bx_t)}{\|\nabla f(\bx_t)\|}.
	\]
	The inclusions in Proposition~\ref{prop:algorithm_conditional_subg} embed
	these pre- and post-direction fields into a single filtration, and
	Proposition~\ref{prop:nonvanishing_gradients} makes \(\ba_t\) well-defined
	almost surely. By
	Lemma~\ref{lem:weighted_scan_bounds}, with probability at least
	\(1-\delta\), the lower scan
	\eqref{eq:weighted_lower_scan_loglog} and the prefix upper control
	\eqref{eq:weighted_prefix_upper_loglog} hold.
	Since the present \(\Gamma_T(\delta)\) dominates the scan confidence factor,
	we may replace \(\Gamma_T^{\rm scan}(\delta)\) by \(\Gamma_T(\delta)\) in the
	two resulting bounds.

	The lower scan gives, for every interval,
	\[
		\sum_{t=a}^{b}w_t\zeta_t
		\ge
		\frac{1}{2d}\sum_{t=a}^{b}w_t
		-
		\frac{C}{dT_0}\Gamma_T(\delta).
	\]
	Setting \(a=k+1\) and \(b=K-1\), then multiplying by \(2d\) and
	exponentiating, gives, simultaneously for every \(1\le K\le T\) and every
	\(k=-1,0,\dots,K-2\),
	\[
		\exp\left(
		-2d\sum_{t=k+1}^{K-1}
		\frac{\zeta_t}{T_t}
		\right)
		\le
		\exp\left(
		-\sum_{t=k+1}^{K-1}\frac1{T_t}
		+\frac{2C}{T_0}\Gamma_T(\delta)
		\right).
	\]
	The case \(k=-1\) is exactly \(\cE_{\rho_K}\), while
	\(k=0,\dots,K-2\) gives \(\cE_{\rho_{K,k}}\). The remaining case
	\(k=K-1\) holds because the left-hand side is \(1\) and
	\(\rho_{K,K-1}\ge1\).
	The prefix upper control gives
	\[
	d\sum_{t=0}^{K-1}w_t\zeta_t
	\le
	C(1+\Lambda)+
	\frac{C}{T_0}\Gamma_T(\delta)
	\le
	C\Lambda+C\frac{\Gamma_T(\delta)}{T_0},
	\qquad K\le T,
	\]
	which is the event \(\cE_\rho^+\).
\end{proof}

\medskip

\begin{proof}[Proof of Lemma~\ref{lem:dec_2}]
	Set
	\(\eta_t = 4d/[\mu (t+T_0) \norm{\bu_t}^2]\) with
	\(T_0 = 32 dL/\mu\), so that
	\(\eta_t \leq 1/(8L\norm{\bu_t}^2)\). Combining this choice with
	Eq.~\eqref{eq:dec_1} gives
	\begin{align*}
		\Delta_{t+1}
		\leq&
		\left(1 - \frac{2d}{T_t}\zeta_t \right)\cdot \Delta_t\\
		&
		+ \frac{4d }{\mu (t+T_0) \norm{\bu_t}^2} \cdot \dotprod{\nabla f(\bx_t), \bu_t\bu_t^\top \be_t}\\
		&
		+
		\frac{32L d^2 }{\mu^2 (t+T_0)^2 \norm{\bu_t}^4}\norm{\bu_t}^2|\bu_t^\top \be_t|^2
		+ \Delta_{\alpha, t}.
	\end{align*}

	Unrolling the preceding recursion with
	\[
	\Pi_{K,k}
	:=
	\prod_{t=k+1}^{K-1}
	\left(
	1-\frac{2d}{T_t}\zeta_t
	\right),
	\]
	yields
	\begin{align*}
		&\Delta_K
		\leq
		\prod_{t=0}^{K-1}
		\left(
		1-\frac{2d}{T_t}\zeta_t
		\right)
		\cdot
		\Delta_0
		\\
		&
		+
		\sum_{k=0}^{K-1} \left( \Pi_{K,k} \cdot \frac{4d \cdot \dotprod{\nabla f(\bx_k), \bu_k\bu_k^\top \be_k}}{\mu (k+T_0) \norm{\bu_k}^2} \right)
		\\
		&
		+
		\sum_{k=0}^{K-1} \left( \Pi_{K,k}
		\cdot
		\left(\frac{32 Ld^2 }{\mu^2 (k+T_0)^2} \cdot \frac{|\bu_k^\top \be_k|^2}{\norm{\bu_k}^2} + \Delta_{\alpha, k}
		\right)
		\right)
	\end{align*}

	Furthermore, Lemma~\ref{lem:g_decom} gives
	\(|\beta| \leq (L\alpha/2)\norm{\bu}^2\), and hence
	\begin{align*}
		\Delta_{\alpha, t}
		=&
		\frac{\eta_t \beta_t^2}{2}
		+ L\eta_t^2 \beta_t^2 \norm{\bu_t}^2
		\leq
		\frac{\eta_t L^2 \alpha^2 \norm{\bu_t}^4}{8}
		+
		\frac{\eta_t^2 L^3 \alpha^2\norm{\bu_t}^6}{4}
		\\
		=&
		\frac{ d L^2\alpha^2 \norm{\bu_t}^2}{2\mu(t+T_0)}
		+
		\frac{ 4d^2 L^3\alpha^2 \norm{\bu_t}^2}{ \mu^2 (t+T_0)^2 }.
	\end{align*}

	By the standing assumption \(L\ge\mu\),
	\(T_0=32dL/\mu\ge32d\). Hence each factor in \(\Pi_{K,k}\) is
	nonnegative, and \(1-x\le e^{-x}\) gives
	\[
		\Pi_{K,k}
		\le
		\exp\left(
		-2d\sum_{t=k+1}^{K-1}
		\frac{\zeta_t}{T_t}
		\right)
		\le
		\rho_{K,k}.
	\]
	The same argument with the product over \(t=0,\dots,K-1\), together with
	\(\cE_{\rho_K}\), bounds the initial product by \(\rho_K\).
	On the events \(\cE_{\rho_K}\) and \(\cE_{\rho_{K,k}}\), we use this
	upper bound only for the nonnegative initial, quadratic, and smoothing
	terms. The linear term is signed and therefore keeps the actual product
	\(\Pi_{K,k}\):
	\begin{align*}
		\Delta_K
		\leq&
		\rho_K \cdot \Delta_0
		+ \sum_{k=0}^{K-1} \Pi_{K,k} \left( \frac{4d \cdot \dotprod{\nabla f(\bx_k), \bu_k\bu_k^\top \be_k}}{\mu (k+T_0) \norm{\bu_k}^2} \right)
		\\
		&+
		\sum_{k=0}^{K-1} \rho_{K,k} \left(\frac{32Ld^2 }{\mu^2 (k+T_0)^2} \cdot \frac{|\bu_k^\top \be_k|^2}{\norm{\bu_k}^2}
		+
		\frac{ d L^2\alpha^2 \norm{\bu_k}^2}{2\mu(k+T_0)}
		+
		\frac{ 4d^2 L^3\alpha^2 \norm{\bu_k}^2}{ \mu^2 (k+T_0)^2 }
		\right).
	\end{align*}
\end{proof}

\subsection{Proofs for the Sub-Gaussian Consequences}

\begin{proof}[Proof of Lemma~\ref{lem:subg_projection}]
	If \(\bz=0\), the claim is immediate. Otherwise, taking
	\(\lambda=1/(2\sigma^2)\) in the assumed norm bound and using
	\(|\langle\be,\bz\rangle|\le\|\be\|\|\bz\|\) gives
	\[
		\EE\left[
		\exp\left(
		\frac{\langle\be,\bz\rangle^2}{2\sigma^2\|\bz\|^2}
		\right)
		\middle|\mathcal F
		\right]
		\le e^{1/2}.
	\]
	Thus \(\langle\be,\bz\rangle\) has a conditional \(\psi_2\)-norm bounded by
	a universal multiple of \(\sigma\|\bz\|\). For completeness, one may pass
	from this conditional Orlicz bound to an mgf bound by conditioning on
	\(\mathcal F\) and applying the usual moment argument pointwise: if
	\(X=\langle\be,\bz\rangle\) and \(a=C\sigma\|\bz\|\), then
	\(\EE[|X|^m\mid\mathcal F]\le C a^m m^{m/2}\) for all \(m\ge2\). Expanding
	\(\EE[\exp(\theta X)\mid\mathcal F]\), using
	\(\EE[X\mid\mathcal F]=0\), and summing the resulting series gives
	\(\EE[\exp(\theta X)\mid\mathcal F]\le
	\exp(C\theta^2a^2)\). Absorbing the universal constants into
	\(C_{\mathrm{sg}}\) yields Eq.~\eqref{eq:subg_projection_mgf}. This is the
	Euclidean specialization of Lemma~2.1 in \citet{liu2023revisiting}.
\end{proof}

\medskip

\begin{proof}[Proof of Lemma~\ref{lem:subg_random_direction_square}]
	If \(\be=0\), the displayed exponential moment equals one. We may
	therefore assume \(\be\ne0\) in the first part of the proof.
	Write \(\bv=\bu/\|\bu\|\). Conditional on \(\be\), the random variable
	\[
		\Theta
		:=
		\frac{(\bv^\top\be)^2}{\|\be\|^2}
	\]
	belongs to \([0,1]\) and satisfies
	\(\EE[\Theta\mid\be,\mathcal F]=1/d\). For \(s\ge0\), convexity of
	\(x\mapsto e^{sx}\) on \([0,1]\) gives
	\[
		e^{s\Theta}
		\le
		1+\Theta(e^s-1).
	\]
	Therefore,
	\[
		\EE_{\bu}\left[
		\exp\left(
			\lambda\frac{(\bu^\top\be)^2}{\|\bu\|^2}
		\right)
		\middle|\be,\mathcal F
		\right]
		\le
		1+\frac{
			\exp(\lambda\|\be\|^2)-1
		}{d}.
	\]
	Taking conditional expectation with respect to \(\be\), using the
	conditional exponential-moment hypothesis of the lemma, and then
	\(1+x\le e^x\), we obtain
	\[
		\EE\left[
		\exp\left(
			\lambda\frac{(\bu^\top\be)^2}{\|\bu\|^2}
		\right)
		\middle|\mathcal F
		\right]
		\le
		\exp\left(
			\frac{e^{\lambda\sigma^2}-1}{d}
		\right).
	\]
	When \(0<\lambda\le1/(2\sigma^2)\), \(e^{\lambda\sigma^2}-1\le
	2\lambda\sigma^2\), proving Eq.~\eqref{eq:subg_random_direction_square_mgf}.

	For the deterministic weighted statement, the conclusion is immediate when
	\(w_{\max}=0\). Otherwise,
	apply Eq.~\eqref{eq:subg_random_direction_square_mgf}
	conditionally at each time with \(\lambda w_k\) in place of \(\lambda\).
	If \(0<\lambda\le 1/(2\sigma^2w_{\max})\), then
	\[
		\EE\exp\left(
			\lambda
			\sum_{k=0}^{K-1}w_k
			\frac{(\bu_k^\top\be_k)^2}{\|\bu_k\|^2}
			-
			\frac{2\lambda\sigma^2}{d}W
		\right)
		\le 1.
	\]
	Markov's inequality with
	\(\lambda=1/(2\sigma^2w_{\max})\) yields
	Eq.~\eqref{eq:subg_weighted_quadratic}.
\end{proof}

\subsection{Proofs for the Linear Martingale Bound}

\begin{proof}[Proof of Lemma~\ref{lem:future_angle_innovations}]
	We first identify the conditional law of the future angles and then use its
	deterministic form to prove conditional independence from \(\xi_k\).

	\proofstep{Conditional product law.}
	Let \(\mathcal F_t^-\) be the history immediately before sampling
	\(\bu_t\). Conditional on \(\mathcal F_t^-\), the iterate \(\bx_t\) is
	known. Hence, on the probability-one event from
	Proposition~\ref{prop:nonvanishing_gradients},
	\[
		\ba_t:=\frac{\nabla f(\bx_t)}{\|\nabla f(\bx_t)\|}
	\]
	is a fixed unit vector. The new direction \(\bu_t\sim\mathcal N(0,I_d)\)
	is independent of \(\mathcal F_t^-\). Therefore, by rotational invariance
	and Lemma~\ref{lem:beta_projection}, for every bounded Borel function \(h\),
	\[
		\EE[h(\zeta_t)\mid \mathcal F_t^-]
		=
		\int h\,d\nu,\qquad
		\nu=\mathrm{Beta}\left(\frac12,\frac{d-1}{2}\right).
	\]
	Crucially, the right-hand side is deterministic: it does not depend on
	\(\bx_t\), even though \(\bx_t\) may depend on earlier oracle samples such
	as \(\xi_k\).

	If \(k=K-1\), there are no future angle variables and the claim is the empty
	product law. Otherwise set
	\[
		Z:=(\zeta_{k+1},\dots,\zeta_{K-1}),\qquad n:=K-k-1.
	\]
	For bounded Borel functions \(h_{k+1},\dots,h_{K-1}\), successive
	conditioning backward from time \(K-1\) gives
	\begin{align*}
	\EE\left[
	\prod_{t=k+1}^{K-1}h_t(\zeta_t)
	\middle|\mathcal F_k^+
	\right] 
	& =
	\EE\left[
	\prod_{t=k+1}^{K-2}h_t(\zeta_t)
	\EE\left[h_{K-1}(\zeta_{K-1})\middle|\mathcal F_{K-1}^{-}\right]
	\middle|\mathcal F_k^+
	\right] \\
	& =
	\left(\int h_{K-1}\,d\nu\right)
	\EE\left[
	\prod_{t=k+1}^{K-2}h_t(\zeta_t)
	\middle|\mathcal F_k^+
	\right].
	\end{align*}
	Repeating the same step for \(t=K-2,\dots,k+1\) yields
	\[
		\EE\left[
		\prod_{t=k+1}^{K-1}h_t(\zeta_t)
		\middle|\mathcal F_k^+
		\right]
		=
		\prod_{t=k+1}^{K-1}\int h_t\,d\nu.
	\]
	By the monotone-class theorem, this product-function identity implies that
	for every bounded Borel function \(h\) on \([0,1]^n\),
	\[
		\EE[h(Z)\mid\mathcal F_k^+]
		=
		\int h\,d\nu^{\otimes n}.
	\]
	In other words,
	\[
		\mathcal L(Z\mid\mathcal F_k^+)=\nu^{\otimes n}.
	\]
	This is stronger than the marginal laws alone: conditional on the
	post-update history \(\mathcal F_k^+\), the future angles are conditionally
	independent and have a deterministic joint law.

	\proofstep{Independence from the current oracle noise.}
	Let \(h\) be any
	bounded Borel function of \(Z\), and let \(\varphi\) be any bounded
	measurable function of \(\xi_k\). Since \(\mathcal F_k^+\) contains
	\(\xi_k\), the variable \(\varphi(\xi_k)\) is
	\(\mathcal F_k^+\)-measurable. Thus the tower property and the deterministic
	conditional law above give
	\begin{align*}
	\EE\!\left[\varphi(\xi_k)h(Z)\mid\mathcal F_k^{u}\right]
	=
	\EE\!\left[
	\varphi(\xi_k)\EE[h(Z)\mid\mathcal F_k^+]
	\middle|\mathcal F_k^{u}
	\right] 
	=
	\left(\int h\,d\nu^{\otimes n}\right)
	\EE[\varphi(\xi_k)\mid\mathcal F_k^{u}].
	\end{align*}
	On the other hand, applying the same tower property without
	\(\varphi(\xi_k)\) gives
	\[
		\EE[h(Z)\mid\mathcal F_k^{u}]
		=
		\int h\,d\nu^{\otimes n}.
	\]
	Combining the last two displays,
	\[
	\EE\!\left[\varphi(\xi_k)h(Z)\mid\mathcal F_k^{u}\right]
	=
	\EE[\varphi(\xi_k)\mid\mathcal F_k^{u}]
	\EE[h(Z)\mid\mathcal F_k^{u}].
	\]
	This identity for all bounded measurable \(h\) and \(\varphi\) is precisely
	the conditional independence of \(\sigma(Z)\) and \(\xi_k\) given
	\(\mathcal F_k^{u}\).

	Finally, \(\be_k=\nabla f(\bx_k)-\nabla f(\bx_k;\xi_k)\) is a measurable
	function of \((\bx_k,\xi_k)\), while \(\bx_k\) is already
	\(\mathcal F_k^{u}\)-measurable. Hence the same conditional independence
	holds with \(\be_k\) in place of \(\xi_k\). Equivalently, for every bounded
	measurable \(\psi\),
	\[
		\EE\!\left[\psi(\be_k)\mid
		\mathcal F_k^{u}\vee\sigma(Z)\right]
		=
		\EE\!\left[\psi(\be_k)\mid\mathcal F_k^{u}\right]
		\quad\text{a.s.}
	\]
	By monotone convergence, the same identity extends to nonnegative
	measurable \(\psi\). Therefore, revealing the future-angle
	\(\sigma\)-field preserves the conditional mean-zero property, the
	conditional law of \(\be_k\), and its sub-Gaussian bounds.
\end{proof}

\medskip

\begin{proof}[Proof of Lemma~\ref{lem:product_linear_uniform}]
	\proofstep{Angle-enlarged filtration.}
	Since \(T_0=32dL/\mu\ge32d\) and \(0\le\zeta_t\le1\), the factors
	\(q_t\) belong to \([15/16,1]\). Hence they satisfy the product-factor
	condition in Lemma~\ref{lem:stitched_product_mg}.
	Use the angle-enlarged pre-noise filtration
	\[
	\mathcal H_k
	:=
	\mathcal F_k^{u}\vee\sigma(\zeta_{k+1},\dots,\zeta_{T-1}),
	\qquad 0\le k\le T-1,
	\qquad
	\mathcal H_T:=\mathcal F_T^{-},
	\]
	where \(\mathcal F_k^{u}\) is the post-direction history in
	Proposition~\ref{prop:algorithm_conditional_subg}. The family
	\((\mathcal H_k)_{k=0}^{T}\) is increasing. Indeed, for \(0\le k\le T-2\),
	\[
	\mathcal F_k^{u}
	\subseteq
	\mathcal F_k^{+}
	\subseteq
	\mathcal F_{k+1}^{-}
	\subseteq
	\mathcal F_{k+1}^{u},
	\]
	and \(\zeta_{k+1}\) is \(\mathcal F_{k+1}^{u}\)-measurable, while
	\(\sigma(\zeta_{k+2},\dots,\zeta_{T-1})\) is part of both adjacent enlarged
	fields. Thus \(\mathcal H_k\subseteq\mathcal H_{k+1}\). For the last step,
	\(\mathcal H_{T-1}=\mathcal F_{T-1}^{u}\subseteq
	\mathcal F_{T-1}^{+}\subseteq\mathcal F_T^{-}=\mathcal H_T\). With respect
	to this filtration, \(q_k\), \(P_{k+1}^{-1}\), and \(v_k\) below are
	\(\mathcal H_k\)-measurable, while \(\vartheta_k\) is
	\(\mathcal H_{k+1}\)-measurable. By
	Lemma~\ref{lem:future_angle_innovations}, revealing the future scalar
	angles does not change the conditional mean-zero property or the
	conditional sub-Gaussian law of the current noise \(\be_k\).

	\proofstep{Martingale increment and variance proxy.}
	Set
	\[
		\vartheta_k
		:=
		\frac{1}{T_k}
		\frac{\nabla f(\bx_k)^\top\bu_k\,\bu_k^\top\be_k}{\|\bu_k\|^2},
		\qquad
		v_k
		:=
		C_{\mathrm{sg}}
		\frac{1}{T_k^2}
		\frac{(\nabla f(\bx_k)^\top\bu_k)^2}{\|\bu_k\|^2}.
	\]
	Then \(\EE[\vartheta_k\mid\mathcal H_k]=0\), and
	Lemma~\ref{lem:subg_projection} gives, for every \(\lambda\in\mathbb R\),
	\[
		\EE\!\left[
		\exp(\lambda\vartheta_k)\mid\mathcal H_k
		\right]
		\le
		\exp(\lambda^2\sigma^2v_k).
	\]
	Since \(\Pi_{K,k}=P_K/P_{k+1}\),
	\[
		\sum_{k=0}^{K-1}Y_{K,k}^{\Pi}
		=
		P_K\sum_{k=0}^{K-1}\frac{\vartheta_k}{P_{k+1}},
		\qquad
		P_K^2\sum_{k=0}^{K-1}\frac{v_k}{P_{k+1}^2}
		=
		C_{\mathrm{sg}}A_K^\Pi.
	\]

	\proofstep{Stitched product boundary.}
	Apply Lemma~\ref{lem:stitched_product_mg} with failure probability
	\(\delta_{\rm lin}\), confidence factor \(G\), deterministic product
	envelope
	\[
		B_T=(T+T_0)^{C_P}
		\exp\!\left(C_P\frac{\Gamma_T(\delta)}{T_0}\right),
	\]
	and deterministic variance envelopes \(C_{\mathrm{sg}}\bar A_K\). The
	displayed dyadic-scale condition remains valid under this fixed rescaling of
	\(\bar A_K\), after a possible enlargement of the universal
	constant. Absorbing \(C_{\mathrm{sg}}\) and the universal
	product-bound constant into \(C_Y\) yields
	Eq.~\eqref{eq:product_linear_uniform_bound}.
\end{proof}

\subsection{Proofs for the Quadratic Noise Term}

\begin{proof}[Proof of Lemma~\ref{lem:Q_up}]
	Fix \(K\le T\). Set
	\[
		w_k=\frac{\rho_{K,k}}{T_k^2},
		\qquad
		W=\sum_{k=0}^{K-1}w_k,
		\qquad
		w_{\max}
		=
		\max_{0\le k\le K-1}w_k.
	\]
	The weights are deterministic and nonnegative, since \(\rho_{K,k}\) is the
	deterministic envelope defined in Eq.~\eqref{eq:rho_kk}.
	If \(w_{\max}=0\), then all \(w_k=0\), and the defining inequality of
	\(\cE_{Q_K}\) is immediate. Hence assume \(w_{\max}>0\) below.

	Let \(\mathcal F_k^{-}\) be the history immediately before sampling
	\(\bu_k\). Conditional on \(\mathcal F_k^{-}\), the iterate \(\bx_k\) is
	fixed. By Proposition~\ref{prop:algorithm_conditional_subg}, the noise
	vector \(\be_k=\nabla f(\bx_k)-\nabla f(\bx_k;\xi_k)\) satisfies
	\[
		\EE[\be_k\mid\mathcal F_k^{-}]=0,
		\qquad
		\EE\!\left[
		\exp\left(\lambda\|\be_k\|^2\right)
		\middle|\mathcal F_k^{-}
		\right]
		\le
		\exp(\lambda\sigma^2),
		\qquad 0<\lambda<\sigma^{-2}.
	\]
	Moreover, the sampling rule of Algorithm~\ref{alg:SA_1} makes
	\(\bu_k\sim\mathcal N(0,I_d)\) conditionally independent of
	\(\mathcal F_k^{-}\vee\sigma(\be_k)\). Hence the one-step bound in
	Lemma~\ref{lem:subg_random_direction_square} applies at time \(k\). For
	any \(\lambda>0\) such that \(\lambda w_k\le1/(2\sigma^2)\),
	\[
		\EE\!\left[
		\exp\left(
		\lambda w_k
		\frac{(\bu_k^\top\be_k)^2}{\|\bu_k\|^2}
		\right)
		\middle|\mathcal F_k^{-}
		\right]
		\le
		\exp\left(\frac{2\lambda\sigma^2}{d}w_k\right).
	\]
	The summands up to time \(k-1\) are \(\mathcal F_k^{-}\)-measurable, so we
	may condition successively backward from \(k=K-1\) to \(k=0\). Iterating
	the preceding conditional-mgf estimate gives, for every
	\[
		0<\lambda\le\frac{1}{2\sigma^2w_{\max}},
	\]
	\[
		\EE\exp\left[
		\lambda
		\sum_{k=0}^{K-1}
		w_k
		\frac{(\bu_k^\top\be_k)^2}{\|\bu_k\|^2}
		-
		\frac{2\lambda\sigma^2}{d}W
		\right]
		\le 1,
	\]
	Markov's inequality with \(\lambda=(2\sigma^2w_{\max})^{-1}\) yields
	\[
	\begin{aligned}
	&\Pr\Bigg(
		\sum_{k=0}^{K-1}
		w_k\frac{(\bu_k^\top\be_k)^2}{\|\bu_k\|^2}
		>
		\frac{2\sigma^2}{d}W
		+
		2\sigma^2w_{\max}\log\frac1{\delta_K^{(Q)}}
	\Bigg)
	\\
	&\qquad\le
	\exp\left(
		-\lambda\,2\sigma^2w_{\max}
		\log\frac1{\delta_K^{(Q)}}
	\right)
	=
	\delta_K^{(Q)}.
	\end{aligned}
	\]
	Substituting \(w_k=\rho_{K,k}/T_k^2\) gives exactly the defining
	inequality of \(\cE_{Q_K}\), and hence
	\(\Pr(\cE_{Q_K})\ge1-\delta_K^{(Q)}\).
\end{proof}

\subsection{Proofs for the Smoothing-Bias Term}

\begin{proof}[Proof of Lemma~\ref{lem:alp_term_up}]
	Apply the Laurent--Massart inequality (Lemma~\ref{lem:chi_all}) to the
	collection of Gaussian coordinates defining \(\|\bu_k\|^2\), with
	coordinate weights
	\(w_k = \rho_{K,k}
	\left(
	\frac{c_1}{T_k} + \frac{c_2}{T_k^2}
	\right)\). The Laurent--Massart middle term is
	\[
	2\sqrt{d\log(1/\delta_K^{(\alpha)})\sum_k w_k^2}.
	\]
	Since \(d\ge1\), this is bounded by
	\[
	2d\sqrt{\log(1/\delta_K^{(\alpha)})\sum_k w_k^2}.
	\]
	Thus, with probability at least \(1-\delta_K^{(\alpha)}\),
	\begin{align*}
		&\sum_{k=0}^{K-1} \rho_{K,k}
		\left(
		\frac{c_1}{T_k} + \frac{c_2}{T_k^2}
		\right)\norm{\bu_k}^2
		\leq
		d \sum_{k=0}^{K-1} \rho_{K,k}
		\left(
		\frac{c_1}{T_k} + \frac{c_2}{T_k^2}
		\right)
		\\
		&+
		2d\sqrt{ \log\frac{1}{\delta_K^{(\alpha)}} \cdot \sum_{k=0}^{K-1} \rho_{K,k}^2
			\left(
			\frac{c_1}{T_k} + \frac{c_2}{T_k^2}
			\right)^2}
		+
		2d \log\frac{1}{\delta_K^{(\alpha)}} \max_{0 \leq k \leq K-1} \rho_{K,k}
		\left(
		\frac{c_1}{T_k} + \frac{c_2}{T_k^2}
		\right).
	\end{align*}
	This implies the stated relaxed event \(\cE_{\alpha_K}\) after choosing
	the universal constant \(C_\chi\) large enough.
\end{proof}

\subsection{Proofs for the Induction Event}

\begin{proof}[Proof of Lemma~\ref{lem:c_rho}]
	By Lemma~\ref{lem:T_sum} with \(k=0\), and by the definition of
	\(\rho_K\) in Eq.~\eqref{eq:rho_k},
	\begin{align*}
		\rho_K
		\le&
		\exp\left(\frac{2C}{T_0}\Gamma_T(\delta)\right)
		\frac{T_0}{T_K}
		\leq
		c_\rho \frac{T_0}{T_K}.
	\end{align*}

	Similarly, by Lemma~\ref{lem:T_sum} and Eq.~\eqref{eq:rho_kk},
	\begin{align*}
		\rho_{K,k}
		&\le
		\exp\left(\frac{2C}{T_0}\Gamma_T(\delta)\right)
		\frac{T_{k+1}}{T_K}.
	\end{align*}
	Since \(T_{k+1}=T_k+1\le2T_k\) for \(T_k\ge1\), we have
	\[
	\rho_{K,k}
	\le
	2\exp\!\left(
	\frac{2C}{T_0}\Gamma_T(\delta)
	\right)
	\frac{T_k}{T_K}
	\leq
	c_\rho
	\frac{T_k}{T_K}.
	\]
\end{proof}

\medskip

\begin{proof}[Proof of Lemma~\ref{lem:time_uniform_controls}]
	\proofstep{Confidence allocation.}
	We combine finitely many time-uniform estimates with finitely many terminal
	estimates. Use the explicit failure levels
\[
	\delta_{\rm scan}
	=
	\delta_{\rm prod}
	:=
	\frac{\delta}{6}
\]
for Lemma~\ref{lem:cE_rho} and Lemma~\ref{lem:product_linear_uniform},
respectively. More explicitly, choose the universal constant \(C\) in the scan
estimates, in \(\rho_K,\rho_{K,k}\), and in \(c_\rho\) large enough so that
\[
	C\,\Gamma_T(\delta)
	\ge
	C_0\left(
	1+\log\frac6\delta+\log J_T+\log(e+T_0)
	\right),
\]
where \(C_0\) is the corresponding universal constant in
Lemma~\ref{lem:weighted_scan_bounds}. Thus applying
Lemma~\ref{lem:cE_rho} at failure level \(\delta/6\) yields exactly the
displayed \(\rho\)-events with \(\Gamma_T(\delta)\) and the enlarged universal
constant \(C\). The terminal events are run with
\(\delta_\star=\exp\{-\Gamma_T(\delta)\}\). Because \(T_0=32dL/\mu\ge32\),
\[
	\delta_\star
	=
	\frac{\delta}
	{e\,J_T\,(e+T_0)}
	\le
	\frac{\delta}{6},
\]
so the two terminal failures contribute at most a fixed fraction of
\(\delta\). A final union bound over
\[
\delta_{\rm scan}+\delta_{\rm prod}+2\delta_\star\le\delta
\]
gives the claimed probability.

The symbols \(\delta_K^{(\cdot)}=\delta_\star\) only set the deterministic
thresholds in the fixed-\(K\) event notation; no union bound over \(K\) is
taken. The weighted-suffix and prefix upper-tail controls follow from
Lemma~\ref{lem:cE_rho}. In particular, since
\(\Gamma_T(\delta)=c_\delta\Lambda\), the prefix event implies
\begin{equation}
\label{eq:prefix_angle_product_envelope}
	d\sum_{t=0}^{K-1}\frac{\zeta_t}{T_t}
	\le
	C\Lambda+C\frac{\Gamma_T(\delta)}{T_0}
	\le
	C\Lambda\left(1+\frac{c_\delta}{T_0}\right),
	\qquad K\le T.
\end{equation}

	\proofstep{Product envelope from the angle scan.}
	For the actual product-linear term, define
\[
	q_t
	:=
	1-\frac{2d}{T_t}\zeta_t,
	\qquad
	P_0:=1,\qquad P_K:=\prod_{t=0}^{K-1}q_t.
\]
Since \(T_0=32dL/\mu\ge32d\), all \(q_t\in[15/16,1]\). The prefix
estimate \eqref{eq:prefix_angle_product_envelope} now gives the required
envelope. Indeed, \(0\le 2d\zeta_t/T_t\le1/16\), and
\(-\log(1-x)\le2x\) for \(0\le x\le1/16\). Hence, for every \(K\le T\),
\[
	\log P_K^{-2}
	=
	-2\sum_{t=0}^{K-1}\log\left(1-\frac{2d\zeta_t}{T_t}\right)
	\le
	4\sum_{t=0}^{K-1}\frac{2d\zeta_t}{T_t}
	=
	8d\sum_{t=0}^{K-1}\frac{\zeta_t}{T_t}
	\le
	C\Lambda+C\frac{\Gamma_T(\delta)}{T_0}.
\]
Exponentiating and enlarging the universal constant gives
\[
	P_K^{-2}
	\le
	(T+T_0)^{C_P}
	\exp\!\left(C_P\frac{\Gamma_T(\delta)}{T_0}\right),
	\qquad K\le T,
\]
for a universal constant \(C_P\). Thus the event \(\mathcal P_T(C_P)\) in
Lemma~\ref{lem:product_linear_uniform} holds on the weighted-scan event, with
the residual prefix-scan correction carried by the explicit
\(\exp\{C_P\Gamma_T(\delta)/T_0\}\) factor.

	\proofstep{Uniform linear-noise boundary.}
	The deterministic envelopes in Eq.~\eqref{eq:Abar_product_linear} are chosen
with \(C_A\) large enough to dominate the normalized-angle transfer used in the
induction proof below.
Let
\[
	U_T:=
	(T+T_0)^{C_P}
	\exp\!\left(C_P\frac{\Gamma_T(\delta)}{T_0}\right)
	\max_{K\le T}\bar A_K.
\]
Since Eq.~\eqref{eq:Abar_product_linear} depends on \(K\) only through
\(T_K^{-2}\), and since \(T_0\le T_K\le T+T_0\), uniformly over \(K\le T\),
\[
\begin{aligned}
	\frac{U_T}{\bar A_K}
	&=
	(T+T_0)^{C_P}
	\exp\!\left(C_P\frac{\Gamma_T(\delta)}{T_0}\right)
	\max_{\ell\le T}\left(\frac{T_K}{T_\ell}\right)^2
	\\
	&\le
	\exp\!\left(C_P\frac{\Gamma_T(\delta)}{T_0}\right)
	(T+T_0)^{C_P}\left(\frac{T+T_0}{T_0}\right)^2
	\\
	&\le
	C\exp\!\left(C_P\frac{\Gamma_T(\delta)}{T_0}\right)
	(T+T_0)^{C_P+2}.
\end{aligned}
\]
The preceding ratio contributes only a double logarithm in \(T+T_0\), plus the
single correction \(\Gamma_T(\delta)/T_0\) inside the inner logarithm. These
terms are controlled by the confidence factor because
\[
	\frac{T(T+T_0)}{T_0}\ge T
	\quad\Longrightarrow\quad
	J_T\ge1+\lceil\log_2 T\rceil,
\]
so \(\log J_T\) controls \(\log\log(e+T)\) up to a universal constant, while
\(\log(e+T_0)\) absorbs the remaining small-\(T\) and offset terms. Also
\(\log(1+\Gamma_T(\delta)/T_0)\le \Gamma_T(\delta)\). Hence the preceding
ratio gives
\[
	\max_{K\le T}
	\log\left(1+\log\left(e+\frac{U_T}{\bar A_K}\right)\right)
	\le
	C\,\Gamma_T(\delta).
\]
Because \(\Gamma_T(\delta)\ge1+\log(1/\delta_{\rm prod})\),
Lemma~\ref{lem:product_linear_uniform} with
\(\delta_{\rm lin}:=\delta_{\rm prod}\) and \(G=\Gamma_T(\delta)\) gives,
outside an additional failure probability \(\delta_{\rm prod}\), the
implication
\[
	A_K^\Pi\le\bar A_K
	\quad\Longrightarrow\quad
	\sum_{k=0}^{K-1}Y_{K,k}^{\Pi}
	\le
		C_Y\sigma\sqrt{\bar A_K\,\Gamma_T(\delta)}
\]
simultaneously for every \(K\le T\). Since
\(\delta_K^{(Y)}=\delta_\star=e^{-\Gamma_T(\delta)}\), this is exactly
\(\mathcal L_K(\bar A_K)\) for all \(K\le T\).

	\proofstep{Terminal nonnegative controls.}
	It remains to place the nonnegative controls required by the induction on
the same event. These controls are used only with terminal index \(T\), so no
terminal-time stitching is needed. Apply
Lemma~\ref{lem:Q_up} and Lemma~\ref{lem:alp_term_up} once, with \(K=T\) and
confidence parameter \(\delta_\star\). Since
\(\log(1/\delta_\star)=\Gamma_T(\delta)\), this gives
\(\cE_{Q_T}\) and \(\cE_{\alpha_T}\) with the logarithmic thresholds recorded
in Eq.~\eqref{eq:c_del}. Their total failure probability is at most
\(2\delta_\star\), which is already included in
the fixed confidence accounting at the start of the proof. Taking the
intersection of this fixed finite family of controls proves the lemma.
\end{proof}

\end{document}